\documentclass[a4paper,10pt]{article}
\usepackage{cmap}
\usepackage[T2A]{fontenc}
\usepackage[utf8]{inputenc}
\usepackage[russian,english]{babel}
\usepackage{url}
\usepackage{enumitem}
\setlist{itemsep=3pt, parsep=3pt}

\usepackage{extarrows}
\usepackage{ulem}

\usepackage{longtable}

\usepackage{subcaption}

\usepackage{tocloft}
\usepackage{amsfonts}
\usepackage{amssymb,amsthm}
\usepackage{xcolor,graphicx}
\usepackage{marginnote}
\usepackage{amsmath}
\usepackage{a4wide}
\usepackage[affil-it]{authblk}

\usepackage{epstopdf}
\usepackage{titlesec}
\titlelabel{\thetitle.\quad}

\usepackage[labelsep=period]{caption}

\usepackage{crossreftools}
\makeatletter
\newcommand{\optionaldesc}[2]{%
	\phantomsection
	#1\protected@edef\@currentlabel{#1}\label{#2}%
}
\makeatother

\numberwithin{equation}{section}

\let\OLDthebibliography\thebibliography
\renewcommand\thebibliography[1]{
	\OLDthebibliography{#1}
	\setlength{\parskip}{1pt}
	\setlength{\itemsep}{1pt plus 0.3ex}
}

\usepackage[colorlinks=true,linktocpage,pdfpagelabels,
bookmarksnumbered,bookmarksopen]{hyperref}
\definecolor{ForestGreen}{rgb}{0.1,0.6,0.05}
\definecolor{EgyptBlue}{rgb}{0.063,0.1,0.6}
\hypersetup{
	colorlinks=true,
	linkcolor=EgyptBlue,         
	citecolor=ForestGreen,
	urlcolor=olive
}

\usepackage[hyperpageref]{backref}

\usepackage{orcidlink}

\allowdisplaybreaks
\usepackage{mathtools}
\mathtoolsset{showonlyrefs}

\usepackage{orcidlink}

\usepackage{accents}

\def\Wo{W_0^{1,p}(\Omega)}

\newcommand{\prnth}[1]{\left(#1\right)}
\newcommand{\quadr}[1]{\left[#1\right]}
\newcommand{\abs}[1]{\left|#1\right|}
\newcommand{\ungulata}[1]{\left\{#1\right\}}
\newcommand{\norm}[1]{\left\|#1\right\|}
\newcommand{\sgn}{\mathop{\mathrm{sgn}}\nolimits}

\newcommand{\I}[2]{\int_{#1}^{#2}}
\newcommand{\Lim}[1]{\lim\limits_{#1}}
\newcommand{\Sum}[2]{\sum\limits_{#1}^{#2}}

\renewcommand{\div}{\mathop{\mathrm{div}}\nolimits}
\newcommand{\testf}{\psi}

\newtheorem{theorem}{Theorem}[section]
\newtheorem{lemma}[theorem]{Lemma}
\newtheorem{proposition}[theorem]{Proposition}
\newtheorem{corollary}[theorem]{Corollary}
\theoremstyle{definition}
\newtheorem{remark}[theorem]{Remark}
\newtheorem{example}[theorem]{Example}

\definecolor{cerise}{RGB}{222, 49, 99}

\usepackage{todonotes}

\newenvironment{algorithm}[1][A]{\vspace{0.25cm} \textbf{Algorithm #1}\vspace{0.1cm}}{\vspace{0.25cm}}

\newcommand{\doi}[1]{\href{http://dx.doi.org/#1}{DOI:#1}}

\title{
	\vspace{-3em}
	Inverse iteration method for higher eigenvalues of the $p$-Laplacian}

\author{Vladimir Bobkov ~\&~ Timur Galimov \\}
\date{}

\AtEndDocument{%
	\par
	\medskip
	\bigskip
	\begin{tabular}{@{}l@{}}%
		(V.~Bobkov)\\[0.2em]
		\textsc{Institute of Mathematics, Ufa Federal Research Centre, RAS}\\
		\textsc{Chernyshevsky str. 112, 450008 Ufa, Russia}\\[0.3em]
		\orcidlink{0000-0002-4425-0218} 0000-0002-4425-0218\\[0.3em]
		\textit{E-mail address}: \texttt{bobkov@matem.anrb.ru, bobkovve@gmail.com}\\[1.5em]
		
		(T.~Galimov)\\[0.2em]
		\textsc{Institute of Mathematics, Ufa Federal Research Centre, RAS}\\
		\textsc{Chernyshevsky str. 112, 450008 Ufa, Russia}\\[0.3em]
		\orcidlink{0000-0002-4425-0218} 0000-0002-4425-0218\\[0.3em]
		\textit{E-mail address}: 
		\texttt{galim.timerxan@yandex.ru}
\end{tabular}}

\begin{document}
	\maketitle
	\vspace*{-5ex}
	\begin{abstract}
		We propose a characterization of a $p$-Laplace higher eigenvalue based on the inverse iteration method with balancing the Rayleigh quotients of the positive and negative parts of solutions to  consecutive $p$-Poisson equations. 
		The approach relies on the second eigenvalue's minimax properties, but the actual limiting eigenvalue depends on the choice of initial function. 
		The well-posedness and convergence of the iterative scheme are proved. 
		Moreover, we provide the corresponding numerical computations.
		As auxiliary results, which also have an independent interest, we provide several properties of certain $p$-Poisson problems. 
		
		\par
		\smallskip
		\noindent {\bf  Keywords}: 
		$p$-Laplacian; second eigenvalue; higher eigenvalues; inverse iteration method; inverse power method. 
		
		\noindent {\bf MSC2010}: 
		35P30,  %Nonlinear eigenvalue problems and nonlinear spectral theory for PDEs
		35J92,	%Quasilinear elliptic equations with $p$-Laplacian
		46-08,  %Computational methods for problems pertaining to functional analysis
		47J10,  %Nonlinear spectral theory, nonlinear eigenvalue problems 
		47J25,  %Iterative procedures involving nonlinear operators
		49R05,  %Variational methods for eigenvalues of operators
		65N25.	%Numerical methods for eigenvalue problems for boundary value problems involving PDEs
	\end{abstract}
	
	%ToC
	\renewcommand{\cftdot}{.}
	\begin{quote}	
		\tableofcontents	
		\addtocontents{toc}{\vspace*{-2ex}}
	\end{quote}
	
	%%%%%%%%%%%%%%%%%%%%%%%%%%

	\section{Introduction}\label{sec:intro}
	
	Let $1 < p < \infty$ and let $\Omega \subset \mathbb{R}^D$ be a domain of finite measure, $D \geqslant 1$.
	Consider the problem of finding an \textit{eigenpair} $(\lambda, u) \in \mathbb{R} \times (W_0^{1,p}(\Omega) \setminus \ungulata{0})$ such that
	\begin{equation}\label{EVP}
		\I{\Omega}{}{} \abs{\nabla u}^{p-2} \langle \nabla u, \nabla \testf \rangle \,dx = \lambda \I{\Omega}{} \abs{u}^{p-2} u \testf \,dx
		\quad \text{for any}~ \testf \in W_0^{1,p}(\Omega).
	\end{equation}
	This is a weak form of the Dirichlet eigenvalue problem for the $p$-Laplace operator $\Delta_p (\cdot)\coloneqq \div (\abs{\nabla (\cdot)}^{p-2}\nabla (\cdot))$, formally stated as:
	$$ 
	\begin{cases}
		-\Delta_p u = \lambda \abs{u}^{p-2} u \text{ in } \Omega, \\
		u\big|_{\partial \Omega} = 0.
	\end{cases}
	$$
	In particular, taking $p = 2$, we get the classical Dirichlet Laplace eigenvalue problem. 
	It is thus natural to call $\lambda$ an \textit{eigenvalue} and the corresponding non-zero solution $u$ of \eqref{EVP} an \textit{eigenfunction} of the $p$-Laplacian.
	Eigenfunctions are precisely critical points of the Rayleigh quotient functional $u \mapsto R[u]$ defined on $\Wo \setminus \{0\}$ as
	$$
	R[u] \coloneqq \frac{\I{\Omega}{} \abs{\nabla u}^p \,dx}{\I{\Omega}{} \abs{u}^p \,dx}, 
	$$ 
	and eigenvalues are critical levels of $R$.
	
	It is known (see, for example, \cite{Lindquist}) that the set of all eigenvalues 
	$$
	\sigma(-\Delta_p) := \ungulata{\lambda \in \mathbb{R}:~ \text{\eqref{EVP} holds for some } u \in W_0^{1,p}(\Omega) \setminus \ungulata{0}}
	$$
	is closed and unbounded from above, while there exists the smallest eigenvalue $\lambda_1 (\Omega, p) > 0$ (the \textit{first eigenvalue}), which can be variationally characterized as 
	\begin{equation}\label{lambda_1 var} 
		\lambda_1(\Omega,p) 
		= 
		\inf_{\substack{u \in W_0^{1,p}(\Omega): \\ u \ne 0}} R[u].
	\end{equation}
	A minimizer of \eqref{lambda_1 var} is the \textit{first eigenfunction}, it is unique up to a constant multiplier, and it is the only eigenfunction that does not change sign in $\Omega$, see, e.g., \cite{Lindquist} for an overview. 
	Any eigenvalue different from the first one is called a \textit{higher} eigenvalue, and any eigenfunction associated with such an eigenvalue is a \textit{higher} eigenfunction. 
	The set of higher eigenvalues also attains its minimum $\lambda_2(\Omega, p) > \lambda_1(\Omega, p)$, called the \textit{second eigenvalue}, which has several variational characterizations, see, e.g., \cite{Anane,Bobkov,BrFr1,Cuesta 1}. 
	Among them, the following one, given in \cite[Proposition~4.2]{Bobkov} (see also \cite[Lemma~3.1]{BrFr1}), will be useful for our purposes: 
	\begin{equation}\label{eq:lambda2}
		\lambda_2 (\Omega, p) = \inf_{\substack{u \in W_0^{1,p}(\Omega): \\ u^{\pm} \ne 0}} \max \ungulata{R[u^+], R[u^-]},
	\end{equation}
	where $u^+$ and $u^-$ are the positive and negative parts of $u$, respectively, defined as 
	$$
	u^+ \coloneqq \max \ungulata{0, u}
	\quad \text{and} \quad 
	u^- \coloneqq \max \ungulata{0, -u},
	$$
	so that $u = u^+ - u^-$. 
	The eigenvalues $\lambda_1 (\Omega,p)$ and $\lambda_2 (\Omega, p)$ are the first two members of an infinite sequence of the so-called \textit{variational eigenvalues} $\{\lambda_k(\Omega,p)\}$, which can be constructed using minimax methods of Ljusternik–Schnirelmann type (see, e.g., \cite{Anane,DrRob}). However, it is not known whether the sequence $\{\lambda_k(\Omega,p)\}$ covers the whole set $\sigma(-\Delta_p)$. This, as well as many other open questions regarding spectral properties of the $p$-Laplacian, makes it particularly important to investigate its eigenpairs not just from analytic, but also from numerical point of view. 
	
	There are plenty of techniques to approximate $\lambda_1 (\Omega, p)$ and the corresponding eigenfunction, see, e.g., \cite{Biezuner,Bozorgnia1,Horak, Hynd,LeftonWei}.  
	Many of these are based on the \textit{inverse iteration} method (also known as the \textit{inverse power} method), which gives rise to a large family of easily implementable and yet effective numerical schemes.
	In the linear case $p=2$, the inverse iteration over the first eigenfunction's orthogonal complement in $\Wo$ allows to approximate \textit{higher} eigenvalues as well, see Section~\ref{sec:inverse-first} for details. 	
	
	In the non-linear case $p \neq 2$, due to a handful of obstacles including the lack of natural orthogonality in $\Wo$, the choice of methods to approximate higher eigenvalues is far less abundant. 
	In this regard, we can only refer to \cite{Horak} for a numerical procedure based on the mountain-pass characterization of $\lambda_2(\Omega,p)$, and to \cite{brounreichel} for the approximation of radially symmetric eigenfunctions using a shooting method. 
	An attempt of generalizing the inverse iteration to higher eigenvalues can be found in \cite{Bozorgnia2}, although it does not seem to be applicable to general domains, see Section~\ref{sec:inverse-first}. 	

	The aim of the present work is to propose an inverse iteration-based scheme that  guarantees the approximation of a higher eigenpair with all the advantages of the underlying method preserved. 
	The proposed approach seems to be new even in the linear case $p=2$. 
	
	This paper is organized as follows.
	In Section~\ref{sec:inverse-first}, we briefly overview the inverse iteration schemes developed for approximating 
	$\lambda_1(\Omega,p)$ and discuss the difficulties arising in the case of higher eigenvalues.
	Section~\ref{sec:inverse-second} contains the main results of the work: we introduce a novel method to obtain higher eigenvalues of the $p$-Laplacian and state its well-posedness and convergence. 
	Section~\ref{sec:properties} is of auxiliary nature, and it concentrates on certain properties of the equation $-\Delta_p v = \abs{f}^{p-2}f$ which we will need in the proofs of our main results. 
	In particular, we provide a useful criterion for an arbitrary $f \in W_0^{1,p}(\Omega)$ to be an eigenfunction.
	In Sections~\ref{sec:well pos} and~\ref{sec:conv}, we prove the well-posedness and convergence of the proposed algorithm, respectively, and note some additional properties.  
	Finally, in Section~\ref{sec:num}, we demonstrate the performance of the algorithm in a series of numerical experiments. 
	Some auxiliary facts regarding general properties of the $p$-Laplacian referred in our proofs are placed in Appendix~\ref{sec:appendix}.

\medskip
\noindent
\textbf{Notation.}
Throughout the text, we denote the $L^p(\Omega)$-norm as $\norm{\cdot}_{p}$ and by $\norm{\nabla u}_{p}$ we mean the $W_0^{1,p}(\Omega)$-norm of $u$, that is,
$$ 
\norm{\nabla u}_{p} \coloneqq \prnth{\I{\Omega}{} \abs{\nabla u}^p \,dx}^{1/p}.
$$
Taking any $u \in L^p (\Omega)$, we always assume that $u$ is a fixed element of the corresponding equivalence class. 
For any non-zero $u \in L^p (\Omega)$, the superscript ``tilde'' will refer to the $L^p(\Omega)$-normalization of $u$, i.e., 
$$
\tilde u \coloneqq \frac{1}{\norm{u}_{p}} \, u.
$$
Note that 
$$
R[u] = R[\tilde u] = \norm{\nabla \tilde u}_{p}^p
\quad \text{for any}~ u \in W_0^{1,p}(\Omega) \setminus \{0\}.$$
For convenience, we write $R^+[u]$ and $R^-[u]$ instead of $R[u^+]$ and $R[u^-]$, if the corresponding Rayleigh quotients are defined. 
For any $u \in \Wo$, we also let $\Omega_u^+$ and $\Omega_u^-$ denote the sets
\begin{equation}\label{nodal sets def}\Omega_u^+ := \ungulata{x \in \Omega\!: u(x) > 0}, \quad \Omega_u^- := \ungulata{x \in \Omega\!: u(x) < 0}. \end{equation}
Note that $u^\pm \in \Wo$, with $\nabla u^\pm = \pm \mathbf{1}_{\Omega_u^\pm} \nabla u$, see, e.g., \cite[Lemma~1.19]{Suomalaiset}.
Here, $\mathbf{1}_K$ stands for the indicator function of a set $K$. 
In particular, it follows from \eqref{EVP} that $R^+[u] = R^-[u]$ for any sign-changing eigenfunction $u$. 

We often consider $p$-Poisson problems of the form
\begin{equation}\label{Dirichlet unformal}
	\I{\Omega}{} \abs{\nabla u}^{p-2} \langle \nabla u, \nabla \testf \rangle \, dx = \I{\Omega}{} f\testf \,dx 
	\quad \text{for any}~ \testf \in W_0^{1,p}(\Omega),
\end{equation}
where $f \in L^{p'}(\Omega)$ and the solution $u$ is of class $\Wo$. 
For the sake of visual clarity, we sometimes formally write \eqref{Dirichlet unformal} as the Dirichlet boundary value problem for the $p$-Poisson equation:
\begin{equation}\label{Dirichlet formal2} 				
	\begin{cases}
		-\Delta_p u = f \text{ in } \Omega, \\
		u\big|_{\partial \Omega} = 0,
	\end{cases}
\end{equation}
always meaning that \eqref{Dirichlet formal2} is understood in the weak sense. 
Moreover, since we deal with solutions in $\Wo$ only, we occasionally omit the boundary conditions in \eqref{Dirichlet formal2}, for short. For properties of the $p$-Poisson problem \eqref{Dirichlet unformal},  see Section~\ref{sec:properties} and Appendix~\ref{sec:appendix}.

\section{Brief overview of the inverse iteration method}\label{sec:inverse-first}

As has already been mentioned above, many algorithms for the numerical approximation of the first $p$-Laplace eigenpair are based on the inverse iteration technique.
Given an initial guess $u_0 \in \Wo \setminus \{0\}$, such algorithms build a functional sequence $\ungulata{u_k}$ by consecutively solving $p$-Poisson problems of the form
\begin{equation}\label{InvIt general} \begin{cases}
		-\Delta_p u_{k+1} = f_k (u_k) \text{ in } \Omega, \\
		u_{k+1}\big|_{\partial \Omega} = 0.
\end{cases}\end{equation}
Alongside $\ungulata{u_k}$, a numerical sequence $\ungulata{\mu_k}$
is generated for the approximation of an eigenvalue according to some rule (for instance, $\ungulata{\mu_k}$ can be defined as the Rayleigh quotients sequence $\ungulata{R[u_k]}$). 
The term $f_k$ in \eqref{InvIt general} usually takes the form $f_k (u_k) = C_k |u_k|^{p-2} u_k$, where $C_k > 0$ are certain constant multipliers (in the simplest case, $C_k \equiv 1$). 
Note that if $-\Delta_p u_{k+1} = C_k |u_k|^{p-2} u_k$ in $\Omega$,
then $-\Delta_p (C_k^{{1}/{(1-p)}} u_{k+1}) = |u_k|^{p-2} u_k$ in $\Omega$, and hence,  be an initial guess $u_0$ fixed, all the sequences $\ungulata{u_k}$ produced by different choice strategies for $\ungulata{C_k}$ are equivalent up to rescaling.
The proper choice of $\ungulata{C_k}$, however, can make the convergence analysis of $\ungulata{u_k}$ and $\ungulata{\mu_k}$ easier, for instance, by preventing the sequence $\ungulata{u_k}$ from approaching zero. 

It is worth noting that inverse iteration schemes have not only computational, but purely theoretical importance as well. 
For instance, under the name of monotone iteration methods, they are extensively used in the fixed point theory \cite{sat}. 
Another abstract application of \eqref{InvIt general} arises if we let $\Omega$ be a bounded domain on a complete Riemannian manifold and take $p = 2$, $u_0 = 1$.
In this case, it is known  \cite{brownian} that for any $k \geqslant 0$ the solution $u_{k+1}$ of 
\begin{equation}\label{geodesic} 
	\begin{cases}
		-\Delta u_{k+1} = (k+1) u_{k} \text{ in } \Omega, \\
		u_{k+1} \big|_{\partial \Omega} = 0, 
\end{cases} \end{equation}
maps each $x \in \Omega$ to a $(k+1)$-th moment (in a probabilistic sense) of the first exit time of the Brownian motion starting at $x$. 
The sequence of Poisson problems \eqref{geodesic} is usually referred to as \textit{the Poisson hierarchy}. 
It is not surprising that the scheme \eqref{geodesic} is related to a specific variational characterization of the first eigenvalue and is often employed directly for its calculation \cite{geodesic,HMP,MM1}.

In the linear case $p = 2$, the inverse iteration technique is a classical tool for approximating Laplace eigenvalues described in many numerical methods textbooks, see, e.g., \cite[Chapter~9]{Sun}. In its simplest form, the scheme can be stated as
\begin{equation}\label{eq:Sun}
	\begin{cases}
		-\Delta u_{k+1} = u_k \text{ in } \Omega, \\
		u_{k+1}\big|_{\partial \Omega} = 0.
	\end{cases} 
\end{equation}
The scheme \eqref{eq:Sun} is designed to approximate the \textit{first} Laplace eigenvalue under appropriate assumptions on the initial guess $u_0$. 
Moreover, the convergence to higher eigenpairs is also  achievable by the fact that $L^2 (\Omega)$ and $W_0^{1,2}(\Omega)$, being the Hilbert spaces, can be endowed with an orthogonal basis of Laplace eigenfunctions. 
Let
$$
\lambda_1 < \lambda_2 \leqslant \ldots \leqslant \lambda_n \leqslant \ldots
$$
be the spectrum of the Laplace operator in $\Omega$ counted with multiplicity, and let 
$\ungulata{e_n}$ be an orthogonal system of the corresponding eigenfunctions.
One can prove that if $u_0 \in \ungulata{e_1, \dots, e_{m-1}}^{\perp} \setminus \ungulata{e_m}^{\perp}$,
then the sequence $\ungulata{\tilde u_k}$ generated by \eqref{eq:Sun} converges to some scaling of $e_m$, see \cite{bessa} and also \cite[Section~6]{Biezuner}. 

The problem of approximating higher eigenpairs can therefore be easily solved in the linear case.
Such a technique, however, is impossible to adapt for $p \ne 2$.
In the non-linear case, most of the proposed schemes aim at approximating only the first eigenvalue $\lambda_1 (\Omega, p)$ along with the corresponding eigenfunction.
In fact, it is proven in \cite{Hynd} that for any $u_0 \in L^p (\Omega)$ the sequence
$$ 
\ungulata{\lambda_1^{\frac{k}{p-1}} (\Omega, p) \cdot u_k}, 
$$
where $\ungulata{u_k}$ is generated by the scheme \eqref{InvIt general} with $f_k (u_k) = \abs{u_k}^{p-2} u_k$,
converges strongly in $\Wo$ either to the first eigenfunction or to the identical zero;
in the case of sufficiently regular $\Omega$, it follows from Hopf's lemma that the non-trivial convergence can be assured by taking $u_0 > 0$.
It is also noted that in this case the sequences $\ungulata{R[u_k]}$ and $\{\|u_{k-1}\|_{p} / \|u_k\|_{p}\}$ are both non-increasing.
The results of \cite{Hynd} were inspired by \cite{Biezuner},
where similar conclusions were made within a slightly different theoretical framework.  
In particular, the authors of \cite{Biezuner} were interested in achieving a \textit{uniform} convergence towards an eigenfunction, and proposed a few other numerical sequences potentially approximating the first eigenvalue.
In \cite{Bozorgnia1}, a somewhat more straightforward convergence result was formulated for the choice $f_k (u_k) = R[u_k] \cdot |\tilde u_k|^{p-2} \tilde u_k$.
We also refer to \cite{ercole, Hynd FA} for generalizations of the inverse iteration to wider classes of elliptic operators in the abstract framework,
and to \cite{doubly nonlinear} for a dynamic interpretation of this approach. 

The inverse iteration technique thus provides a relatively well examined characterization of the first $p$-Laplace eigenpair for any value of $p > 1$. 
As has been already noted above, its use for investigating higher eigenvalues is, by contrast, very limited. 
In \cite{Bozorgnia2}, an iterative scheme similar to that of \cite{Bozorgnia1} was proposed for computing the second eigenpair: given a sign-changing $u_0 \in W_0^{1,p}(\Omega)$, construct a sequence $\ungulata{u_k}$ by consecutively solving \eqref{InvIt general} with
\begin{equation}\label{eq:boz1}
f_k (u_k) = R^+ [u_k] \frac{(u_k^+)^{p-1}}{\norm{u_k^+}_{p}^{p-1}} - R^- [u_k] \frac{(u_k^-)^{p-1}}{\norm{u_k^-}_{p}^{p-1}},
\end{equation}
see \cite[Algorithm~1]{Bozorgnia2}. 
Note that here $u_k^+$ and $u_k^-$ are normalized separately. 
This algorithm might converge to a higher eigenpair when $\Omega$ and the initial guess $u_0$ are chosen to have certain symmetries. However, in general, it cannot be guaranteed that the produced sequence $\{u_k\}$ preserves its initial sign-changing property and that $\{R[u_k]\}$ converges to a \textit{higher} eigenvalue. 
The scheme \eqref{InvIt general} with $f_k$ defined by  \eqref{eq:boz1} served as the initial  source of inspiration for \hyperref[alg]{Algorithm A} which we introduce in the next section.

\section{Main results}\label{sec:inverse-second}

We propose the following inverse iteration scheme, which, as will be stated and proved further, allows to approximate a higher eigenpair of the $p$-Laplace operator in $\Omega$.

\begin{algorithm}\label{alg}
\begin{enumerate}[label={\arabic*.}]
\item Choose an initial guess $u_0 \in L^\infty(\Omega)$ 
such that			
\begin{equation}\label{eq:u0-signs}
	\min\ungulata{\norm{u_0^+}_{p}, \norm{u_0^-}_{p}} > 0 \quad \text{and} \quad \Big|\{x \in \Omega:\, u_0(x) = 0\}\Big| = 0.
\end{equation}
\item For $k = 0, 1, 2, \dots$:
\begin{enumerate}[label=(\roman*)]
	\item\label{alg:step2}  Define the mapping $\varphi_k (\alpha)$: $[0;1] \to \Wo$ by assigning to $\alpha \in [0;1]$ the solution $\varphi_k(\alpha)$ of the $p$-Poisson problem
	\begin{equation}\label{phi}
		\begin{cases}
			-\Delta_p \varphi_k (\alpha) = \alpha \prnth{\tilde u_k^+}^{p-1} - \beta(\alpha)\cdot \prnth{\tilde u_k^-}^{p-1} \text{ in } \Omega, \\
			\varphi_k(\alpha)\big|_{\partial \Omega} = 0,
		\end{cases}
	\end{equation}
	where $\beta(\alpha)$: $[0;1] \to [0;1]$ is a continuous monotonically decreasing function with $\beta(0) = 1$ and $\beta(1) = 0$,  independent of $k$. For instance, $\beta(\alpha) = 1 - \alpha$ or $\beta(\alpha) = \sqrt[p]{1-\alpha^p}$.
	
	\item Find any root 
	$\alpha_k \in (0;1)$ 
	of the equation 
	\begin{equation}\label{eq:R+=R-}
		R^+ [\varphi_k(\alpha)] 
		= 
		R^- [\varphi_k (\alpha)]
	\end{equation}
	and let $\beta_k := \beta(\alpha_k)$.
	
	\item Set $u_{k+1} := \varphi_k (\alpha_k)$, increase the iterations counter $k$ by $1$ (i.e., set $k \coloneqq k + 1$), and return to Step~\ref{alg:step2}. 
\end{enumerate}
\end{enumerate}
\end{algorithm}

\noindent	
In other words, we adjust the coefficient $\alpha$ of \eqref{phi} to equate the Rayleigh quotients of the positive and negative parts of the solution of \eqref{phi} at every iteration.

\medskip
We formulate the well-posedness, monotonicity, and convergence of the proposed algorithm in the following two statements.
\begin{theorem}[Well-posedness and monotonicity]\label{th:existence-of-alphaK}
Let $k \geqslant 0$. 
Then there exists a root $\alpha_k \in (0;1)$ of the equation \eqref{eq:R+=R-}. 
As a consequence, $u_{k+1} := \varphi_{k} (\alpha_{k})$ is defined, it is sign-changing, solves the $p$-Poisson problem   
\begin{equation}\label{phi2}
\begin{cases}
	-\Delta_p u_{k+1} = \alpha_k \prnth{\tilde u_k^+}^{p-1} - \beta_k\, \prnth{\tilde u_k^-}^{p-1} \text{ in } \Omega, \\
	u_{k+1}\big|_{\partial \Omega} = 0,
\end{cases}
\end{equation}
and satisfies
\begin{equation}\label{eq:RRR}
R[u_{k+1}] 
=
R^{+}[u_{k+1}]
=
R^{-}[u_{k+1}].
\end{equation} 
Moreover, we have
\begin{equation}\label{prop:monotonicity}
\lambda_2(\Omega,p)
\leqslant
R[u_{k+2}] 
\leqslant
R[u_{k+1}],
\end{equation}
where $R[u_{k+2}] = R[u_{k+1}]$ if and only if $u_{k+1}$ is an eigenfunction. 
\end{theorem}

\begin{theorem}[Convergence]\label{thm:main}
Let $\{u_k\}$ be the sequence defined by \hyperref[alg]{Algorithm A} and $\,{\cal U}$ be the set of its strong limit points in $W_0^{1,p}(\Omega)$. Then the following assertions hold:
\begin{enumerate}[label={\rm(\Roman*)}]
\item\label{thm:main:R is an eigenvalue} 
The sequence $\ungulata{R[u_k]}$ monotonically decreases towards a higher eigenvalue $R^*$: 
\begin{equation}\label{R*}
	R^* := \inf\limits_{k \geqslant 1} R[u_k] = \lim\limits_{k \to \infty}R[u_k]. 
\end{equation}
\item\label{thm:main:partial limits are eigenfunctions} ${\cal U} \ne \varnothing$ and it is a subset of the (sign-changing) eigenfunctions corresponding to the eigenvalue $R^*$.
\item\label{thm:main:shifts} Any subsequence $\ungulata{u_{k_n}}$ contains a sub-subsequence having a strong limit in $\Wo$. 
If $\ungulata{u_{k_n}}$ converges strongly to some $u \in {\cal U}$, then, for any $i \in \mathbb{N}$, the shifted subsequence $\ungulata{u_{k_n + i}}$ also converges strongly to $u$.
\item\label{thm:main:dist to 0} $\Lim{k \to \infty} \rho(u_k, {\cal U}) = 0$, where $\rho$ stands for the distance function in $W_0^{1,p}(\Omega)$.
\item\label{thm:main:alphas} There exists a limit
$$ 
\Lim{k \to \infty} \alpha_k = \alpha^* = \Lim{k \to \infty} \beta_k, 
$$
coinciding with the only fixed point of the mapping $\beta(\cdot)$.
\item\label{thm:main:structure of U} ${\cal U}$ either consists of a single eigenfunction or has no isolated elements. 
In particular, if $R^*$ is a simple eigenvalue, then the whole sequence $\ungulata{u_k}$ converges to the corresponding eigenfunction. 
\end{enumerate}
\end{theorem}

We refer to Proposition~\ref{lem:if any subsequence is separated from zero} and Lemma~\ref{lem:U} for some further properties of $\{u_k\}$. 
Let us also provide a few additional comments.

\begin{remark}\label{uniqueness}
The ``index shift'' property \ref{thm:main:shifts} and the overall convergence of $\ungulata{\alpha_k}$ may indicate that even if the eigenvalue $R^*$ is not simple, the whole sequence $\ungulata{u_k}$ still converges to some higher eigenfunction, i.e.,  $\mathop{\mathrm{card}}\nolimits\, {\cal U} = 1$.
We discuss difficulties in verificating this claim in Remark~\ref{remark:counterexample}. 
Note that regardless of whether ${\cal U}$ is a singleton or not, it is guaranteed by Statement~\ref{thm:main:dist to 0} that \hyperref[alg]{Algorithm A} always provides numerically reliable results: no matter at which step $k$ large enough we interrupt the iterative scheme, the resulting approximation $u_k$ lies sufficiently close to the set of  eigenfunctions corresponding to $R^*$. 
\end{remark}

\begin{remark}\label{rem:initial-guess}
Let us provide a few examples illustrating the ways in which the value of $R^*$ depends on the choice of the initial guess $u_0$. 

\begin{enumerate}
\item If $u_0 \in \Wo$ is chosen to be a higher eigenfunction associated with an eigenvalue $\mu = R[u_0]$, then it is easy to see that  \hyperref[alg]{Algorithm A} with the choice $\alpha_k = \alpha^*$ for $k \geqslant 0$ preserves the Rayleigh quotient and converges to a particular scaling of $u_0$. 
Indeed, at the very first iteration, the problem \eqref{phi} with $\alpha_0 = \alpha^* = \beta(\alpha_0)$ gives
$$ 
-\Delta_p \varphi_0 (\alpha^*) = \alpha^* \abs{\tilde u_0}^{p-2} \tilde u_0 = -\Delta_p \quadr{\prnth{\frac{\alpha^*}{\mu}}^{\frac{1}{p-1}} \tilde u_0}, 
$$
so that $\alpha^*$ solves \eqref{eq:R+=R-} and therefore we can let
$$ 
u_1 = \prnth{\frac{\alpha^*}{\mu}}^{\frac{1}{p-1}} \tilde u_0, \quad \norm{u_1}_{p} = \prnth{\frac{\alpha^*}{\mu}}^{\frac{1}{p-1}}. 
$$
Note that $\alpha^*$ delivers the solution of \eqref{eq:R+=R-} regardless of whether $u_0$ satisfies the second assumption in \eqref{eq:u0-signs} or not. 
Arguing by induction and taking $\alpha_k = \alpha^*$ at each step $k \geqslant 1$, we get
\begin{gather}  
	u_2 = \prnth{\frac{\alpha^*}{\mu}}^{\frac{1}{p-1}} \tilde u_1 = \prnth{\frac{\alpha^*}{\mu}}^{\frac{1}{p-1}} \frac{u_1}{\norm{u_1}_{p}} = u_1 = \prnth{\frac{\alpha^*}{\mu}}^{\frac{1}{p-1}} \tilde u_0,\\
	\vdots\\
	u_{k+1} = \prnth{\frac{\alpha^*}{\mu}}^{\frac{1}{p-1}} \tilde u_k = u_1 = \prnth{\frac{\alpha^*}{\mu}}^{\frac{1}{p-1}} \tilde u_0.
\end{gather}
We thus conclude that if $u_0$ is a higher eigenfunction, then the choice $\alpha_k = \alpha^*$ for $k \geqslant 0$ results in \hyperref[alg]{Algorithm A} converging to a particular scaling of $u_0$ with $R^* = \mu$. 

\item\label{rem:initial-guess:2} In practice, it might be useful to consider $\Omega$ and $u_0$ having a certain symmetry (e.g., both being radially symmetric). 
In this case, by induction and due to the uniqueness of the solution to the $p$-Poisson problem \eqref{phi2}, each $u_k$ inherits the imposed symmetry, and hence so do their limits.
Therefore, we can approximate ``symmetric'' eigenfunctions and corresponding ``symmetric'' eigenvalues by choosing an appropriate initial guess $u_0$. 
See Section~\ref{sec:num} for some numerical results and discussion. 

\item Suppose that the second eigenvalue  $\lambda_2(\Omega,p)$ is isolated in $\sigma(-\Delta_p)$, that is, there exists $\varepsilon>0$ such that
$$
\Big(\lambda_2(\Omega,p)-\varepsilon;  \lambda_2(\Omega,p)+\varepsilon\Big) \, \cap \, \sigma(-\Delta_p) = \{\lambda_2(\Omega,p)\}. 
$$	
Then, given a sign-changing initial guess $u_0$ with $R[u_0] < \lambda_2(\Omega,p) + \varepsilon$, \hyperref[alg]{Algorithm A} will converge to the second eigenpair, as it follows from Theorem~\ref{thm:main} \ref{thm:main:R is an eigenvalue}. 
However, the required isolatedness assumption does not seem to be proved for $p \neq 2$ and $D \geqslant 2$. 
Note that such a reasoning can be applied only to the second eigenvalue specifically: for any other isolated higher eigenvalue $\lambda > \lambda_2 (\Omega, p)$ choosing $u_0$ such that $\lambda < R[u_0] < \lambda + \varepsilon$ for $\varepsilon > 0$ small enough might be insufficient to ensure the convergence towards $\lambda$. 
\end{enumerate}

It is important to note that the constrained mountain-pass algorithm of \cite{Horak} has the same particularities as discussed in the examples above. Namely, although it is designed to approximate $\lambda_2(\Omega,p)$, its actual convergence to the \textit{second} eigenvalue seems to be difficult to establish rigorously. At the same time, the algorithm of \cite{Horak} is also able to approximate ``symmetric'' eigenpairs.
\end{remark}

\begin{remark}
The assumption $\Big|\{x \in \Omega:\, u_0(x) = 0\}\Big| = 0$  imposed on $u_0$ in \hyperref[alg]{Algorithm A} is used only in the proof of Theorem~\ref{th:existence-of-alphaK}. 
It would be interesting to know whether this property of the nodal set $\ungulata{u_0(x)=0}$ is necessary for the existence of roots $\alpha_k$ ($k \geqslant 0$) in \eqref{eq:R+=R-}. 
We also note that the $p$-Poisson problem \eqref{phi} and the result of Theorem~\ref{th:existence-of-alphaK} are 	somewhat related to the problem studied in \cite{bergbucur}, where estimates on $\alpha_k$ were obtained in the case $p=2$ for $\tilde u_k = \mathbf{1}_{\Omega \setminus A} - \mathbf{1}_{A}$, with $A$ being some measurable subset of $\Omega$. 
\end{remark}

\section{Properties of the equation $-\Delta_p v = \abs{f}^{p-2}f$}\label{sec:properties}

In order to prove our main results regarding  \hyperref[alg]{Algorithm A}, we start by examining properties of the $p$-Poisson problem with the right-hand side taking the form $|f|^{p-2} f$ for some $f \in W_0^{1,p}(\Omega) \setminus \ungulata{0}$:
\begin{equation}\label{vf0} 
\I{\Omega}{} \abs{\nabla v}^{p-2} \langle \nabla v, \nabla \testf \rangle \,dx = \I{\Omega}{} \abs{f}^{p-2}f \, \testf \,dx \quad \text{for any } \testf \in W_0^{1,p}(\Omega).
\end{equation}
It is not hard to see that $|f|^{p-2} f \in L^{p'}(\Omega)$, and hence \eqref{vf0} has a unique solution $v \in W_0^{1,p}(\Omega) \setminus \ungulata{0}$. 
Since
$$\abs{f}^{p-2} f = \abs{f}^{p-2} (f^+ - f^-) = (f^+)^{p-1} - (f^-)^{p-1}, $$
we can also rewrite \eqref{vf0} as
\begin{equation}\label{vf} 
\I{\Omega}{} \abs{\nabla v}^{p-2} \langle \nabla v, \nabla \testf \rangle \,dx = \I{\Omega}{}  \quadr{(f^+)^{p-1} - (f^-)^{p-1}} \testf \,dx \quad \text{for any } \testf \in W_0^{1,p}(\Omega).
\end{equation}

\begin{remark}
The problem \eqref{phi}, which we base our iterative scheme on, can be seen as a particular case of \eqref{vf} with 
$$
v = \varphi_k (\alpha), \quad 
f^+ = \alpha^{\frac{1}{p-1}} \, \tilde u_k^+,
\quad 
f^- = \beta^{\frac{1}{p-1}} (\alpha) \, \tilde u_k^-.
$$
Moreover, problems of the type \eqref{vf} will later arise in Section~\ref{sec:conv}, playing an important role in the convergence analysis of \hyperref[alg]{Algorithm A}.
\end{remark}

Let $v$ solve \eqref{vf0} for some $f \ne 0$. 
It is known that $R[v] \leqslant R[f]$, see \cite[Proposition~2.4]{Hynd}, and also \cite[Lemma~2.1]{Bozorgnia1}, \cite[Lemma~4]{ercole}. 
Note that if $f$ is an eigenfunction, then the corresponding eigenvalue equals $R[f]$, so that the homogeneity of the $p$-Laplacian implies $v = R^{1/(1-p)}[f] \cdot f$ and $R[v] = R[f]$.
In fact, the converse is also true, namely, $f \ne 0$ is an eigenfunction whenever  $R[v] = R[f]$.

\begin{proposition}\label{inequalities system}
Let $f, v \in \Wo \setminus \ungulata{0}$ satisfy \eqref{vf0}. Then $R[v] \leqslant R[f]$. Moreover, $R[v] = R[f]$ if and only if $f$ is an eigenfunction, so that $v = R^{1/(1-p)}[f] \cdot f$.
\end{proposition}
\begin{proof}
Although the monotonicity of the Rayleigh quotients is known in the literature, we provide a proof of this fact for more clarity. 
Let us test the identity \eqref{vf0} with $\psi = v$ and employ the Hölder inequality to get
\begin{equation}\label{lambda_1 shows up in the left sqrt[p]-2 no signs}
\norm{\nabla v}_p^p = \I{\Omega}{} |f|^{p-2} f v \, dx \leqslant \I{\Omega}{} |f|^{p-1} |v| \, dx \leqslant \norm{|f|^{p-1}}_{p'} \norm{v}_p = \norm{f}_p^{p-1} \norm{v}_p. 
\end{equation}
Dividing \eqref{lambda_1 shows up in the left sqrt[p]-2 no signs} by $\norm{v}_p \ne 0$, we obtain
\begin{equation}\label{Right sign L-norm estimate G-2 no signs} 
	R^{1/p}[v] \cdot \norm{\nabla v}_p^{p-1} \leqslant \norm{f}_p^{p-1}. 
\end{equation}
On the other hand, by testing \eqref{vf0} with $\psi = f$ we have
$$ \norm{f}_p^p = \I{\Omega}{} \abs{\nabla v}^{p-2} \langle \nabla v, \nabla f \rangle \, dx \leqslant \norm{\nabla v}_p^{p-1} \norm{\nabla f}_p = R^{1/p}[f] \cdot \norm{\nabla v}_p^{p-1} \norm{f}_p,
$$
and therefore
\begin{equation}\label{eq:estim0 G-0 no signs}  \norm{f}_p^{p-1} \leqslant R^{1/p}[f] \cdot \norm{\nabla v}_p^{p-1}. 
\end{equation}
Combining \eqref{Right sign L-norm estimate G-2 no signs} with \eqref{eq:estim0 G-0 no signs}, we deduce $R[v] \leqslant R[f]$.

Next, suppose $R[v] = R[f]$. In this case, the inequalities \eqref{Right sign L-norm estimate G-2 no signs} and \eqref{eq:estim0 G-0 no signs} become equalities, which in turn makes \eqref{lambda_1 shows up in the left sqrt[p]-2 no signs} a chain of equalities:
\begin{equation}\label{lambda_1 shows up in the left sqrt[p]-2 no signs eq} 
	\I{\Omega}{} |f|^{p-2} fv \, dx = \I{\Omega}{} |f|^{p-1} |v| \, dx = \norm{|f|^{p-1}}_{p'} \norm{v}_p. 
\end{equation}
The second equality in \eqref{lambda_1 shows up in the left sqrt[p]-2 no signs eq} 
yields $|v| = C|f|$ for some $C > 0$, while  the first one implies that $\sgn v = \sgn f$. We thus conclude $v = Cf$. 
Furthermore, by the equality in \eqref{eq:estim0 G-0 no signs}, we have $C = R^{1/(1-p)} [f]$. 
In view of \eqref{EVP}, $f$ is an eigenfunction corresponding to the eigenvalue $R[f] = C^{1-p}$.

Finally, the case of $f$ being an eigenfunction leads to the equality $R[v] = R[f]$, as has been already examined prior the statement. 
\end{proof}

The result of Proposition~\ref{inequalities system} will play an important role in our analysis, but considering only the values of $R[v]$ and $R[f]$ is insufficient for dealing with the approximation of higher eigenvalues, as the variational characterization \eqref{eq:lambda2} of $\lambda_2 (\Omega, p)$ indicates. Therefore, we need to draw some information about the positive and negative parts of $v$ and $f$, too, and the rest of this section will be concerned with their relation.

Taking $\testf = f^+$ in \eqref{vf} and applying the Hölder inequality,  we get
\begin{equation}\label{Right sign L-norm estimate G}
\norm{f^+}_{p}^p \leqslant \norm{\nabla v}^{p-1}_{p}  \norm{\nabla f^+}_{p},
\end{equation} 
and, in the same manner for $\testf = -f^-$,
\begin{equation}\label{Right sign L-norm estimate G-2}
\norm{f^-}_{p}^p \leqslant \norm{\nabla v}^{p-1}_{p}  \norm{\nabla f^-}_{p}.
\end{equation}
Next, for $\testf = v^+$, \eqref{vf} gives
$$ 
\I{\Omega}{} \abs{\nabla v^+}^{p} \,dx = \I{\Omega}{} v^+ \, (f^+)^{p-1} \, dx - \I{\Omega}{} v^+ \, (f^-)^{p-1} \, dx. 
$$
Since the subtrahend is non-negative, we obtain
\begin{equation}\label{eq:estim0 G-0}
\norm{\nabla v^+}_{p}^p \leqslant \I{\Omega}{} v^+ \, (f^+)^{p-1} \, dx
\leqslant \norm{v^+}_{p}  \norm{f^+}_{p}^{p-1}.
\end{equation}
In the case $v^+ \ne 0$, dividing \eqref{eq:estim0 G-0} by $\norm{v^+}_{p}$ and extracting $R^+[v]$, we get
\begin{equation}\label{lambda_1 shows up in the left sqrt[p]} 
\prnth{R^+ [v]}^{1/p} \cdot \norm{\nabla v^+}_{p}^{p-1} \leqslant \norm{f^+}_{p}^{p-1}. 
\end{equation}
Similarly, taking $\testf = -v^-$ in \eqref{vf}, we deduce that 
\begin{equation}
\label{eq:estim0 G-2}
\norm{\nabla v^-}_{p}^p \leqslant \I{\Omega}{} v^- \, (f^-)^{p-1} \, dx \leqslant \norm{v^-}_{p}  \norm{f^-}_{p}^{p-1},
\end{equation}
and the assumption $v^- \ne 0$ further leads to
\begin{equation}\label{lambda_1 shows up in the left sqrt[p]-2} 
\prnth{R^- [v]}^{1/p} \cdot \norm{\nabla v^-}_{p}^{p-1} \leqslant  \norm{f^-}_{p}^{p-1}. 
\end{equation}

\smallskip 
Suppose now that the solution $v$ of \eqref{vf0} does not change sign in $\Omega$, for instance, $v \geqslant 0$. If $f$ is sign-constant as well, then by Lemma~\ref{lem:sign-change lemma} we have $v = v^+$ and $f = f^+$, so that the inequality $R[v] \leqslant R^+[f]$ follows from Proposition~\ref{inequalities system}. 
We further state that this inequality holds even for $f$ being sign-changing, and formulate the following counterpart of Proposition~\ref{inequalities system}.

\begin{proposition}\label{lemma:main:2}
Let $f, v \in W_0^{1,p}(\Omega) \setminus \{0\}$ satisfy \eqref{vf} and $v$ be 
sign-constant in $\Omega$.
Then $R[v] \leqslant R^+[f]$ if $v \geqslant 0$, and $R[v] \leqslant R^-[f]$ if $v \leqslant 0$. Moreover, the equality is achieved if and only if $v = \lambda_1^{1/(1-p)} (\Omega, p) \cdot f$ is the first eigenfunction.
\end{proposition}
\begin{proof}
Assume, without loss of generality, that $v$ is non-negative. Then Lemma~\ref{lem:sign-change lemma} yields $f^+ \ne 0$. 	
To deduce that $R[v] \leqslant R^+[f]$, we combine \eqref{Right sign L-norm estimate G} with \eqref{lambda_1 shows up in the left sqrt[p]} to obtain 
\begin{equation}\label{:>+} 
\prnth{R^+ [v]}^{1/p} \cdot \norm{\nabla v^+}_{p}^{p-1} \leqslant \norm{f^+}_{p}^{p-1} \leqslant \norm{\nabla v}_{p}^{p-1} \cdot \prnth{R^+[f]}^{1/p}.
\end{equation}
Since $v = v^+ \ne 0$, the inequality $R[v] \leqslant R^+ [f]$ follows. 

Assuming $R[v] = R^+[f]$, the inequality \eqref{:>+} becomes an equality that can be rewritten as
$$ 
\norm{f^+}_{p}^{p-1} \norm{v}_p = \norm{\nabla v}_{p}^{p-1} R^{1/p}[v] \cdot \norm{v}_p =
\norm{\nabla v}_p^p,
$$
turning \eqref{eq:estim0 G-0} into the following Hölder equality:
$$ 
\I{\Omega}{} (f^+)^{p-1} \, v \, dx
= \norm{(f^+)^{p-1}}_{p'} \norm{v}_p. 
$$
Thus, there exists $C > 0$ such that $v = Cf^+$. Taking $\testf = -f^-$ in \eqref{vf}, we get
$$ \norm{f^-}_{p}^p = -C^{p-1} \I{\Omega}{} \abs{\nabla f^+}^{p-2} \langle \nabla f^+, \nabla f^- \rangle \, dx = 0, 
$$
and therefore $f$ is non-negative: $f = f^+ = C^{-1} v$. Substituting $f$ by $C^{-1} v$ in \eqref{vf}, we see that $v$ is the first eigenfunction and that $\lambda_1(\Omega,p) = C^{1-p} = R[v]$. 

Finally, if $v$ is the first eigenfunction, then it is obvious that $v = \lambda_1^{1/(1-p)} (\Omega, p) \cdot f$ with $R[v] = R^+[f] = \lambda_1 (\Omega,p)$.
\end{proof}

The case of both $v$ and $f$ being sign-changing and satisfying $R^+[v] = R^-[v]$ and $R^+ [f] = R^-[f]$ will be particularly important for our analysis. 
In this case, the algebraic inequality 
\begin{equation}\label{fracrat}
\min_{i=\overline{1,2}} \frac{a_i}{b_i} \leqslant \frac{a_1 + a_2}{b_1 + b_2} \leqslant \max_{i=\overline{1,2}}\frac{a_i}{b_i} \quad \text{for any}~ a_i, b_i > 0, \, i = \overline{1,2}, 
\end{equation}
implies that $R[v] = R^{\pm}[v]$ and $R[f] = R^{\pm}[f]$, so it follows from Proposition~\ref{inequalities system} and  the variational characterization \eqref{eq:lambda2} that
\begin{equation}\label{RRl}
\lambda_2 (\Omega, p)
\leqslant
R[v] = R^{\pm}[v] 
\leqslant
R[f] = R^{\pm}[f].
\end{equation}

\section{Well-posedness and monotonicity}\label{sec:well pos}

In this section, we prove 
Theorem~\ref{th:existence-of-alphaK}, arguing by induction. 
For this purpose, we introduce the following assumption: 
\begin{enumerate}[label={($\mathcal{A}$)}]
\item\label{assumption} Either $k=0$ or after $k \geqslant 1$ iterations we have successfully obtained roots $\alpha_0$, $\dots$, $\alpha_{k-1} \in (0;1)$ of the equations
$$
R^- [\varphi_j(\alpha_j)] 
= 
R^+ [\varphi_j(\alpha_j)], \quad j = \overline{0,k-1},
$$ 
so that the corresponding solutions $u_1 = \varphi_0(\alpha_0)$, $u_2 = \varphi_1(\alpha_1)$, $\dots$, $u_k = \varphi_{k-1} (\alpha_{k-1})$ of \eqref{phi} are defined and all change sign in $\Omega$.
\end{enumerate}

In the following three lemmas, we provide a few properties of solutions $\varphi_k(\alpha)$ to the problem~\eqref{phi}.

\begin{lemma}\label{lem:contin}
Let \ref{assumption} hold. 
Then the mappings
\begin{align}
\label{lem:contin:1}
\alpha \mapsto \norm{\nabla \varphi_k(\alpha)}_{p}
\quad &\text{and} \quad 
\alpha \mapsto \norm{\varphi_k(\alpha)}_{p},\\
\label{lem:contin:2}
\alpha \mapsto \norm{\nabla \varphi_k^{\pm}(\alpha)}_{p}
\quad &\text{and} \quad 
\alpha \mapsto \norm{\varphi_k^{\pm}(\alpha)}_{p}
\end{align}
are continuous in $[0;1]$.
\end{lemma}
\begin{proof}
The continuity of the mappings \eqref{lem:contin:1} follows from Corollary~\ref{right sides convergence yields the solutions convergence} and the Poincar\'e inequality. 
Hence, the mappings \eqref{lem:contin:2} have the same property in view of, e.g., \cite[Lemma~1.22]{Suomalaiset}. 
\end{proof}

\begin{lemma}\label{lem:monotone}
Let \ref{assumption} hold. 
Then $\varphi_{j}(\alpha) \in C^1(\Omega) \cap L^\infty(\Omega)$ for any $j \leqslant k$. 
Moreover, for any $0 \leqslant t_1 < t_2 \leqslant 1$ we have  $\varphi_k(t_1) \leqslant \varphi_k (t_2)$, so that
\begin{align}
\Omega_{\varphi_k(t_1)}^+ \subset \Omega_{\varphi_k(t_2)}^+ \quad \text{and} \quad \norm{\varphi_k^+(t_1)}_{p} \leqslant \norm{\varphi_k^+(t_2)}_{p}, \\
\Omega_{\varphi_k(t_1)}^- \supset \Omega_{\varphi_k(t_2)}^- \quad \text{and} \quad \norm{\varphi_k^-(t_1)}_{p} \geqslant \norm{\varphi_k^-(t_2)}_{p}, 
\end{align}
where $\Omega_{\varphi_k (\alpha)}^{\pm}$ are defined as in \eqref{nodal sets def}. 
\end{lemma}
\begin{proof}
Since the initial guess $u_0$ is required to be of class $L^{\infty}(\Omega)$, we have
$$
\alpha \prnth{\tilde u_0^+}^{p-1} - \beta(\alpha) \cdot \prnth{\tilde u_0^-}^{p-1}
\in 
L^{\infty}(\Omega)
\quad \text{for any}~ \alpha \in [0;1].
$$
The differentiability results of \cite{DiBenedetto,tolksdorf}, as well as standard bootstrap arguments, assure that
the solution $\varphi_0(\alpha)$ of \eqref{phi} belongs to $C^1(\Omega) \cap L^\infty(\Omega)$ for any $\alpha \in [0;1]$. 
By induction, we have $\varphi_{j}(\alpha) \in C^1(\Omega) \cap L^\infty(\Omega)$ for any $j \leqslant k$. 

Since $\alpha \mapsto \beta(\alpha)$ is decreasing, we have $s_1 = \beta(t_1) > \beta(t_2) = s_2$. 
Therefore, 
$$
t_1 (\tilde u_k^+)^{p-1} \leqslant t_2 (\tilde u_k^+)^{p-1}
\quad \text{and} \quad 
s_1 (\tilde u_k^-)^{p-1} \geqslant s_2 (\tilde u_k^-)^{p-1},
$$
which implies that
$$
t_1 (\tilde u_k^+)^{p-1} - s_1 (\tilde u_k^-)^{p-1} \leqslant t_2 (\tilde u_k^+)^{p-1} - s_2 (\tilde u_k^-)^{p-1}.
$$
Due to the weak comparison principle (see, e.g., \cite[Theorem~2.4.1]{PS}), we get $\varphi_k(t_1) \leqslant \varphi_k (t_2)$. 
In particular, $0 \leqslant \varphi_k^+ (t_1) \leqslant \varphi_k^+ (t_2)$, which yields 
$$ 
\Omega_{\varphi_k(t_1)}^+ \subset \Omega_{\varphi_k(t_2)}^+ \quad \text{and} \quad \norm{\varphi_k^+(t_1)}_{p} \leqslant \norm{\varphi_k^+(t_2)}_{p}. 
$$
The assertions for $\varphi_k^-(\cdot)$ are obtained in the same way.
\end{proof}

\begin{lemma}\label{lem: phi nodal sets}
Let \ref{assumption} hold.
Then for any $\alpha \in [0;1]$ the nodal set of $\varphi_j (\alpha)$, $j\leqslant k$, has measure zero, that is, 
\begin{equation}\label{eq:uk-zero-measure}
\Big|\ungulata{x \in \Omega:~ (\varphi_{j}(\alpha))(x) = 0}\Big| = 0 \quad \text{for any } j\leqslant k, ~\alpha \in [0;1].
\end{equation}
\end{lemma}
\begin{proof}
First consider $j = 0$. Recall that $\varphi_0 (\alpha)$ solves \eqref{phi} with $u_0$ satisfying \eqref{eq:u0-signs}.
Then, for $\alpha =0$, the conclusion \eqref{eq:uk-zero-measure} follows from the first assumption in \eqref{eq:u0-signs} and the strong maximum principle: 
$$
-\Delta_p \varphi_0 (0) = -(\tilde u_0^-)^{p-1} \leqslant 0 \quad \text{and} \quad u_0^- \ne 0
\quad \text{imply that} \quad 
\varphi_0 (0) < 0.
$$
The case $\alpha=1$ is covered analogously.     
For $0 < \alpha < 1$, \eqref{eq:uk-zero-measure} follows from the second assumption in \eqref{eq:u0-signs} by the result of \cite[Corollary~1.1]{Lou} (see also \cite[Corollary~1.7]{ACF}). 

Next, suppose $k \geqslant 1$ and consider $j = 1$. Recall that in this case we assume (see \ref{assumption}) that the root $\alpha_0 \in (0;1)$ of the equation $R^+[\varphi_0 (\alpha)] = R^-[\varphi_0 (\alpha)]$ exists, where $\varphi_0 (\alpha_0) =: u_1$ changes sign in $\Omega$ and, as has just been shown, has the nodal set of measure zero. Repeating the above arguments for  $\varphi_1(\alpha)$, we conclude that \eqref{eq:uk-zero-measure} holds for $j = 1$. Continuing by induction, we establish \eqref{eq:uk-zero-measure} for any $j \leqslant k$.  
\end{proof}

Lemmas~\ref{lem:contin}, \ref{lem:monotone},  and \ref{lem: phi nodal sets} allow us to prove the existence of $\alpha_k$ solving the equation \eqref{eq:R+=R-} under the assumption \ref{assumption}.

\begin{proposition}\label{prop:existence-of-alphaK}
Let \ref{assumption} hold. Then there exists $\alpha_k \in (0;1)$ such that $R^+[\varphi_k (\alpha_k)] = R^- [\varphi_k (\alpha_k)]$. 
As a consequence, the sign-changing solution $u_{k+1} = \varphi_k(\alpha_k)$ of \eqref{phi2} is defined. 
\end{proposition}
\begin{proof}
In order to prove the statement, we will demonstrate the existence of  $\alpha_{\min} < \alpha_{\max}$ such that both $R^+[\varphi_k (\alpha)]$ and $R^- [\varphi_k (\alpha)]$ are defined in $(\alpha_{\min}; \alpha_{\max})$ with 
\begin{alignat}{2}\label{limits left}
&R^- [\varphi_k (\alpha_{\min})] < \infty,  
\quad &&R^+ [\varphi_k (\alpha_{\min} + 0)] = \infty, \\ 
\label{limits right}  
&R^- [\varphi_k (\alpha_{\max} - 0)] = \infty, 
\quad &&R^+ [\varphi_k (\alpha_{\max})] < \infty. 
\end{alignat}
If so, the function $F(\alpha) := R^+ [\varphi_k (\alpha)] - R^-[\varphi_k (\alpha)]$ will be continuous in $(\alpha_{\min} ; \alpha_{\max})$ by Lemma~\ref{lem:contin}, at the same time exhibiting the following behavior at the interval's boundaries:
\begin{equation}\label{F(a)} 
F(\alpha_{\min} + 0) = +\infty, \quad F(\alpha_{\max} - 0) = -\infty, 
\end{equation}
which implies the existence of the desired root $\alpha_k \in (\alpha_{\min}; \alpha_{\max})$.

Since for $\alpha = 0$ the right-hand side of \eqref{phi} is non-positive and not identically zero, the strong maximum principle implies $\varphi_k (0) < 0$ in $\Omega$, i.e.,  $\norm{\varphi_k^+ (0)}_{p} = 0$. 
Recalling that the mapping $\alpha \mapsto \norm{\varphi_k^+(\alpha)}_{p}$ is continuous and monotone in $[0;1]$ (see Lemmas~\ref{lem:contin} and \ref{lem:monotone}), 
we can define
\begin{equation}\label{def alpha min}
\alpha_{\min} :=  
\inf\big\{
\alpha \in (0;1):~
\norm{\varphi_k^+ (\alpha)}_{p} > 0
\big\}.
\end{equation}
In particular, we have $\alpha_{\min} \in [0;1)$, $\norm{\varphi_k^+ (\alpha_{\min})}_{p} = 0$, and $\varphi_k^-(\alpha_{\min}) \ne 0$. 
Therefore, $R^-[\varphi_k(\alpha_{\min})]$ is finite, while $R^+[\varphi_k
(\alpha)]$ is undefined for $\alpha \in [0;\alpha_{\min}]$. Let us show, however, that $R^+[\varphi_k (\alpha_{\min}+0)]  =  \infty$.

From Lemma~\ref{lem:monotone}, for any $t_1 > \ldots > t_n > \ldots$ such that $t_n \to \alpha_{\min}$,  there exists a pointwise limit $\phi$ of $\ungulata{\varphi_k (t_n)}$, which by the $L^p(\Omega)$-continuity of $\varphi_k(\alpha)$ must coincide with $\varphi_k (\alpha_{\min})$ a.e.\ in $\Omega$. 
Recall that $\varphi_k (\alpha_{\min}) \leqslant 0$ in $\Omega$. 
Since by Lemma~\ref{lem:monotone} we have
\begin{equation}\label{eq:monotinicity1}
\Omega_{\varphi_k(t_1)}^+ \supset \ldots \supset \Omega_{\varphi_k(t_n)}^+ \supset \ldots ,
\end{equation}
the non-positivity of $\varphi_k(\alpha_{\min})$ and the pointwise monotone convergence $\varphi_k (t_n) \to \varphi_k (\alpha_{\min})$ imply the inclusion
$$ 
\bigcap_{n=1}^{\infty} \Omega_{\varphi_k(t_n)}^+ \subset \ungulata{x \in \Omega:~ (\varphi_{k}(\alpha_{\min}))(x) = 0} \cup \mathcal{N}, 
$$
where the set $\mathcal{N} := \big\{x \in \Omega: \phi(x) \neq (\varphi_{k}\big(\alpha_{\min})\big)(x)\big\}$ has measure zero. 
Therefore, \eqref{eq:uk-zero-measure}, \eqref{eq:monotinicity1}, the arbitrariness of $\{t_n\}$, and the continuity of measure yield
$$		
\abs{\Omega_{\varphi_k(\alpha)}^+} \to 0 \quad \text{as}~ \alpha \to \alpha_{\min}+0.
$$
As noted in  Lemma~\ref{lem:monotone}, the function $\varphi_k (\alpha)$ is continuous in $\Omega$, and hence we conclude $\varphi_k^+(\alpha) \in W_0^{1,p}\big(\Omega_{\varphi_k(\alpha)}^+\big)$  by \cite[Lemma~5.6]{Cuesta 1}. 
Consequently, the Faber-Krahn inequality ensures that
$$ 
R^+[\varphi_k (\alpha)] 
\geqslant \lambda_1 \prnth{\Omega_{\varphi_k(\alpha)}^+, p} \geqslant 
|B|^{p/D} \lambda_1(B,p)\,\abs{\Omega_{\varphi_k(\alpha)}^+}^{-p/D} \to \infty
\quad \text{as}~ \alpha \to \alpha_{\min}+0,
$$
where $B$ is any open ball in $\mathbb{R}^D$.
That is, the proof of \eqref{limits left} is complete.

Now we note that for $\alpha = 1$ the right-hand side of \eqref{phi} is non-negative and not identically zero, and hence the strong maximum principle implies $\varphi_k (1) > 0$ in $\Omega$, i.e.,  $\norm{\varphi_k^- (1)}_{p} = 0$. In the same manner as above, we define
\begin{equation}\label{def alpha max}
\alpha_{\max} :=  
\sup\big\{
\alpha \in (0;1):~
\norm{\varphi_k^- (\alpha)}_{p} > 0
\big\}
\end{equation}
and establish \eqref{limits right}. It is obvious that $\alpha_{\min} < \alpha_{\max}$, since otherwise $\varphi_k(\alpha) = 0$ for $\alpha_{\max} \leqslant \alpha \leqslant \alpha_{\min}$, while the right-hand side of \eqref{phi} is never identically zero.

By the definitions \eqref{def alpha min} and \eqref{def alpha max}, both $\norm{\varphi_k^+(\alpha)}_{p}$ and $\norm{\varphi_k^-(\alpha)}_{p}$ are non-zero for $\alpha \in (\alpha_{\min};\alpha_{\max})$, so that by Lemma~\ref{lem:contin} the functions $R^{\pm}[\varphi_k (\cdot)]$ are continuous in $(\alpha_{\min};\alpha_{\max})$. Then \eqref{F(a)} assures the existence of $\alpha_k \in (\alpha_{\min};\alpha_{\max})$. 
\end{proof}

Given the result of Proposition~\ref{prop:existence-of-alphaK}, we can finally justify the assumption \ref{assumption} for all $k \geqslant 0$ and complete the proof of Theorem~\ref{th:existence-of-alphaK}.

\begin{proof}[Proof of Theorem~\ref{th:existence-of-alphaK}]

Note that for $k = 0$ the assumption \ref{assumption} holds trivially, so that Proposition~\ref{prop:existence-of-alphaK} guarantees the existence of $\alpha_0 \in (0;1)$ solving $R^+[\varphi_0 (\alpha_0)] = R^-[\varphi_0(\alpha_0)]$, and we let $u_1 = \varphi_0 (\alpha_0)$.
The assumption \ref{assumption} is thus satisfied for $k = 1$, implying the existence of $\alpha_1 \in (0;1)$ such that $R^+[\varphi_1 (\alpha_1)] = R^-[\varphi_1 (\alpha_1)]$.
This in turn means that \ref{assumption} holds for $k = 2$.
Continuing by induction, we see that \ref{assumption} is true for any $k \geqslant 0$.
Applying Proposition~\ref{prop:existence-of-alphaK}, we assure the existence of infinite sequences $\ungulata{\alpha_k}$ and $\ungulata{u_k}$, $u_{k+1} = \varphi_k (\alpha_k)$, solving \eqref{eq:R+=R-} and \eqref{phi2}, respectively. 
Taking \eqref{fracrat} into account, we derive \eqref{eq:RRR}. 
Applying \eqref{RRl} to
\begin{equation}
v = u_{k+2},
\qquad 
f =  \alpha_{k+1}^{1/(p-1)} \tilde u_{k+1}^+ - \beta_{k+1}^{1/(p-1)} \tilde u_{k+1}^-, 
\end{equation}
we further conclude \eqref{prop:monotonicity}. 
Finally, Proposition~\ref{inequalities system} states that $R[u_{k+2}] = R[u_{k+1}]$ if and only if $u_{k+1}$ is an eigenfunction.
\end{proof}

As a corollary of Theorem~\ref{th:existence-of-alphaK}, we see that the normalized sequence $\ungulata{\tilde u_k}$ is bounded in $W_0^{1,p}(\Omega)$:
\begin{equation}\label{boundedness for tildas}
\norm{\nabla \tilde u_k}_{p}^p = R[\tilde u_k] \leqslant R[\tilde u_1] < \infty, \quad k \geqslant 1.
\end{equation}
To conclude this section, let us show that the same is true for the unscaled sequence $\ungulata{u_k}$ as well -- we will need this fact later in Section~\ref{sec:conv}.

\begin{lemma}\label{lem:bounds}
The sequence $\{u_k\}$ is bounded in $\Wo$, and, consequently, in $L^p(\Omega)$.
\end{lemma}

\begin{proof}
Recall from Theorem~\ref{th:existence-of-alphaK} that each $u_k$ is sign-changing.
The estimates \eqref{eq:estim0 G-0} and \eqref{eq:estim0 G-2} for $v$, $f$ defined by 
\begin{equation}\label{u_k to f}
v = u_{k+1},
\qquad
f =  \alpha_k^{1/(p-1)} \tilde u_k^+ - \beta_k^{1/(p-1)} \tilde u_k^-,
\end{equation}
read as
\begin{equation}\label{Left sign W-norm estimate} 
\norm{\nabla u_{k+1}^+}_{p}^p 
\leqslant \alpha_k \norm{u_{k+1}^+}_{p} \cdot \norm{\tilde u_k^+}_{p}^{p-1},
\quad 
\norm{\nabla u_{k+1}^-}_{p}^p 
\leqslant \beta_k \norm{u_{k+1}^-}_{p} \cdot \norm{\tilde u_k^-}_{p}^{p-1}.
\end{equation}
Dividing by $\norm{\nabla u_{k+1}^{+}}_{p}$ and $\norm{\nabla u_{k+1}^{-}}_{p}$, respectively, applying \eqref{lambda_1 var}, and recalling that $\ungulata{\alpha_k}, \ungulata{\beta_k} \subset (0;1)$, we get
\begin{equation}\label{murituri}
\norm{\nabla u_{k+1}^+}_{p}^{p-1} \leqslant \lambda_1^{-1/p} (\Omega,p) \norm{\tilde u_k^{+}}_{p}^{p-1},
\quad
\norm{\nabla u_{k+1}^-}_{p}^{p-1} \leqslant \lambda_1^{-1/p} (\Omega,p) \norm{\tilde u_k^{-}}_{p}^{p-1}.
\end{equation}
Raising to the power of $p/(p-1)$ and summing the resulting inequalities, we obtain
$$ 
\norm{\nabla u_{k+1}}_{p}^p \leqslant \lambda^{\frac{1}{1-p}}_1(\Omega,p) \prnth{\norm{\tilde u_k^+}_{p}^p + \norm{\tilde u_k^-}_{p}^p} = \lambda^{\frac{1}{1-p}}_1(\Omega,p), 
$$
which means the desired boundedness of $\{u_{k}\}$ in $\Wo$.
\end{proof}

\section{Convergence}\label{sec:conv}

In this section, we prove Theorem~\ref{thm:main} on the convergence of the sequence $\ungulata{u_k}$ generated by \hyperref[alg]{Algorithm A}. 
The arguments have the following structure. 
First, in Section~\ref{sec:conv1}, we show that the value $R^*$ given by \eqref{R*}
(note that its existence follows immediately from \eqref{prop:monotonicity}) is indeed a higher eigenvalue, and that any strong partial limit of the normalized sequence $\{\tilde u_k\}$ in $\Wo$ is an eigenfunction associated with $R^*$. 
Then we use this result to derive the proof of  Theorem~\ref{thm:main} in Section~\ref{sec:conv:main}.

From now on, we suppose $k \geqslant 1$, since the Rayleigh quotient of an arbitrary initial guess $u_0 \in L^{\infty}(\Omega)$ satisfying \eqref{eq:u0-signs} may be undefined.

Before proceeding with the arguments, we provide a few terminological conventions on our treatment of subsequences. 
On many occasions below, while working with a
sequence $\ungulata{a_m}$, only the properties of some properly chosen subsequence $\ungulata{a_{m_l}}$ will be of our interest. 
In this case, we will sometimes say that we \textit{reduce} or \textit{thin out} the original sequence, meaning that we proceed investigating $\ungulata{a_{m_l}}$ while referring to it simply as $\ungulata{a_m}$. 
For instance, any subsequence $\ungulata{u_{k_n}}$ of $\ungulata{u_k}$, being bounded in $\Wo$ by Lemma~\ref{lem:bounds}, can be reduced in such a way that $u_{k_n} \to u$ weakly in $\Wo$ for some $u$; 
by the Rellich-Kondrachov theorem there is a further thin-out of $\ungulata{u_{k_n}}$ which converges to $u$ strongly in $L^p(\Omega)$, and so on.

\subsection{Preliminary results}\label{sec:conv1}

Since the sequence $\{\tilde u_k\}$ is by construction normalized in $L^p (\Omega)$, and therefore separated from zero, the properties of its limit points are somewhat easier to investigate compared to the unscaled sequence $\{u_k\}$. For the convenience of further analysis, we write the weak form of the problem \eqref{phi2} as 
\begin{equation}\label{renormed} 
\I{\Omega}{}\abs{\nabla \tilde u_{k + 1}}^{p-2} \langle \nabla \tilde u_{k + 1}, \nabla \testf \rangle \,dx = \I{\Omega}{}  \quadr{(\hat u_{k}^+)^{p-1} - (\hat u_{k}^-)^{p-1}} \testf \,dx \quad \text{for any } \testf \in W_0^{1,p}(\Omega), 
\end{equation}
where the non-linear scaling $\hat u_k$ of $u_k$ is defined via
\begin{equation}\label{hats} 
\hat u_{k}^+ \coloneqq \frac{\alpha_{k}^{1/(p-1)}}{\norm{u_{k + 1}}_{p}} \tilde u_{k}^+, \quad \hat u_{k}^- \coloneqq \frac{\beta_{k}^{1/(p-1)}}{\norm{u_{k + 1}}_{p}} \tilde u_{k}^-.
\end{equation}
We can thus view $\{\tilde u_{k+1}\}$ as a sequence of the solutions to the $p$-Poisson equations with the right-hand sides of the form $|\hat u_k|^{p-2}\hat u_k$. 

\begin{lemma}\label{boundedness for hats}
The sequences $\ungulata{\hat u_k^{\pm}}$ are bounded in $\Wo$.
\end{lemma}
\begin{proof}
Consider the sequence of negative parts and observe that the estimate \eqref{Right sign L-norm estimate G-2} with the choice
\begin{equation}
v = u_{k+1}, 
\qquad 
f =  \alpha_k^{1/(p-1)} \tilde u_k^+ - \beta_k^{1/(p-1)} \tilde u_k^-,
\end{equation}
reads as 
\begin{equation}\label{Right sign L-norm estimate}
\beta_k \norm{\tilde u_k^-}_{p}^p \leqslant \norm{\nabla u_{k+1}}_{p}^{p-1} \norm{\nabla \tilde u_k^-}_{p}.
\end{equation}
At the same time, since $R^-[\tilde u_{k}] = R[u_k]$ by Theorem~\ref{th:existence-of-alphaK}, the definition \eqref{hats} and the boundedness result \eqref{boundedness for tildas} give
\begin{equation}
\norm{\nabla \hat u_{k}^-}_{p}^{p-1} 
= 
\frac{\beta_{k}}{\norm{u_{k + 1}}_{p}^{p-1}} \norm{\nabla \tilde u_{k}^-}_{p}^{p-1}
\leqslant 
R^{\frac{p-1}p}[u_1] \, \frac{\beta_{k} \norm{\tilde u_{k}^-}_{p}^{p-1}}{\norm{u_{k + 1}}_{p}^{p-1}}.
\end{equation}
Applying \eqref{Right sign L-norm estimate}, we deduce that
$$
\norm{\nabla \hat u_{k}^-}_{p}^{p-1}
\leqslant R^{\frac{p-1}p}[u_1] \, \frac{\norm{\nabla u_{k + 1}}_{p}^{p-1}}{\norm{u_{k + 1}}_{p}^{p-1}} \cdot \frac{\norm{\nabla \tilde u_{k}^-}_{p}}{\norm{\tilde u_{k}^-}_{p}} 
=
\Big(R[u_1]\cdot R[u_{k + 1}]\Big)^{\frac{p-1}{p}} \, R^{\frac{1}{p}}[\tilde u_{k}^-] 
\leqslant 
R^{\frac{2p-1}{p}}[u_1],
$$
which completes the proof. 
The case of positive parts can be analyzed in the same way.
\end{proof}

We can now apply Corollary~\ref{strong partial limits of solutions} to derive some important convergence properties of the sequence $\{\tilde u_k\}$.

\begin{proposition}\label{lem:conv:sep}
$R^* = \inf R[u_k]$ is a higher eigenvalue.
Moreover, any index sequence $\{k_n\}$ can be reduced in such a way that the sequence $\{\tilde u_{k_n}\}$ converges strongly in $\Wo$ to a higher eigenfunction associated with the eigenvalue $R^*$. 
In particular,
\begin{equation}\label{wapaqat} 
\liminf_{k \to \infty} \norm{\tilde u_{k}^{\pm}}_{p} > 0.  
\end{equation}
\end{proposition}
\begin{proof}
Let $\ungulata{k_n}$ be an increasing sequence in $\mathbb{N}$.
Thanks to Lemma~\ref{boundedness for hats} and Corollary~\ref{strong partial limits of solutions} applied to the problem \eqref{renormed}, we can reduce $\ungulata{k_n}$ in such a way that 
\begin{gather}\label{tilda u weakly} 
\hat u_{k_n - 1} \to \varphi \text{ weakly in } \Wo \text{ and strongly in } L^p(\Omega),\\
\tilde u_{k_n} \to u \text{ strongly in } \Wo, 
\end{gather}
for some $\varphi, u \in \Wo$, where $u$ solves $-\Delta_p u = |\varphi|^{p-2}\varphi$ in $\Omega$. 
Since $\norm{\tilde u_{k_n}}_{p} \equiv 1$, the strong convergence in $\Wo$ yields $\norm{u}_{p} = 1$, so that $u = \tilde u$ and $\varphi \ne 0$.

By Theorem~\ref{th:existence-of-alphaK} and  Proposition~\ref{inequalities system}, we have $R^* = R[u] \leqslant R[\varphi]$.
Note that the homogeneity of $R[\cdot]$ implies
$$ 
R^{\pm} [\hat u_k] = \frac{\norm{\nabla \hat u_k^\pm}_{p}^{p}}{\norm{\hat u_k^\pm}_{p}^p} = \frac{\norm{\nabla \tilde u_k^\pm}_{p}^p}{\norm{\tilde u_k^\pm}_{p}^p} = R[u_k], 
$$
and hence we conclude 
from \eqref{fracrat} and Theorem~\ref{th:existence-of-alphaK} that $R[\hat u_k] = R[u_k] \to R^*$.
On the other hand, the convergence properties of $\ungulata{\hat u_{k_n - 1}}$ and the lower semi-continuity of the norm give
\begin{equation}\label{^_^}
R^* =  
\Lim{n \to \infty} \frac{\norm{\nabla \hat u_{k_n - 1}}_{p}^{p}}{\norm{\hat u_{k_n -1}}_{p}^p}
= \norm{\varphi}_p^{-p}\cdot\Lim{n \to \infty} \norm{\nabla \hat u_{k_n - 1}}_p^p
\geqslant \frac{\norm{\nabla \varphi}_{p}^p}{\norm{\varphi}_{p}^p} = R[\varphi].
\end{equation}
Thus, we derive $R^* = R[u] = R[\varphi]$, 
and Proposition~\ref{inequalities system} guarantees that $\tilde u$ is an eigenfunction associated with the eigenvalue $R[u] = R^* \geqslant \lambda_2 (\Omega,p)$. It is therefore a \textit{higher} eigenvalue, and hence $u$ is sign-changing.  

Suppose now that $\liminf\limits_{k \to \infty} \norm{\tilde u_{k}^{+}}_{p} = 0$. 
Then there exists a subsequence $\{\tilde u_{k_n}^{+}\}$ such that $\tilde u_{k_n}^{+} \to 0$ strongly in $L^p(\Omega)$. 
By the already obtained result, we can further assume that $\{\tilde u_{k_n}\}$ converges strongly in $\Wo$ to some higher eigenfunction $\phi \in \Wo \setminus \{0\}$. 
Since $\phi^+ \neq 0$, we get a contradiction.
The same argument applies to $\liminf\limits_{k \to \infty} \norm{\tilde u_{k}^{-}}_{p}$, establishing \eqref{wapaqat}. 
\end{proof}

Having Proposition~\ref{lem:conv:sep} in hand, we can show that the unscaled sequence $\{u_k\}$ is  separated from zero, which will play a crucial role in our proof of Theorem~\ref{thm:main} in the following section.

\begin{lemma}\label{lem:separatedness}
There exist constants $\eta_W, \eta_L > 0$ such that $\norm{\nabla u_k}_{p}  \geqslant \eta_W$, $\norm{u_k}_{p}  \geqslant \eta_L$.
\end{lemma}
\begin{proof}
Suppose that there exists a subsequence $\ungulata{u_{k_n}}$ of $\{u_k\}$ whose shift $\ungulata{u_{k_n+1}}$ satisfies
\begin{equation}\label{eq:subseq:inf>0}
\Lim{n \to \infty} \norm{\nabla u_{k_n + 1}}_{p} = 0.
\end{equation}
Consider the relation between $\ungulata{u_{k_n + 1}}$ and $\ungulata{\tilde u_{k_n}}$ given by \eqref{phi2}:
\begin{equation}\label{shift by +1}
\I{\Omega}{} \abs{\nabla u_{k_n + 1}}^{p-2} \langle \nabla u_{k_n + 1}, \nabla \testf \rangle \,dx = \I{\Omega}{} \quadr{\alpha_{k_n} (\tilde u_{k_n}^+)^{p-1} - \beta_{k_n} (\tilde u_{k_n}^+)^{p-1}} \testf \,dx.
\end{equation}
Since $\|\tilde u_{k_n}^\pm\|_p \leqslant \norm{\tilde u_{k_n}}_p = 1$ and $\norm{\nabla \tilde u_k^\pm}_{p}^p \leqslant \norm{\nabla \tilde u_k}_{p}^p \leqslant R[u_1] $ (see \eqref{boundedness for tildas}),
the inequalities \eqref{Right sign L-norm estimate G} and \eqref{Right sign L-norm estimate G-2} with the choice
\begin{equation}
v = u_{k_n+1},
\qquad 
f =  \alpha_{k_n}^{1/(p-1)} \tilde u_{k_n}^+ - \beta_{k_n}^{1/(p-1)} \tilde u_{k_n}^-,
\end{equation} give
\begin{equation}\label{eq:conv:altern1}
\max\Big\{\alpha_{k_n} \|\tilde u_{k_n}^+\|_{p}^p, \beta_{k_n} \|\tilde u_{k_n}^-\|_{p}^p\Big\}
\leqslant R^{\frac{1}{p}}[u_1] \cdot {\norm{\nabla u_{k_n + 1}}_{p}^{p-1}} \to 0 \text{ as } n \to \infty.
\end{equation}
In view of \eqref{wapaqat}, we have $\max\{\alpha_{k_n}, \beta_{k_n}\} \to 0$, which contradicts the decreasing property of $\beta(\cdot)$.
Consequently, there exists $\eta_W > 0$ such that $\norm{\nabla u_k}_{p}  \geqslant \eta_W$. 
In view of \eqref{boundedness for tildas}, we get 
$$
\eta_W \leqslant \norm{\nabla u_k}_{p} = R^{\frac{1}{p}}[u_k] \cdot \norm{u_k}_{p} 
\leqslant 
R^{\frac{1}{p}}[u_1] \cdot \norm{u_k}_{p}, 
$$
which implies the existence of $\eta_L > 0$ such that $\norm{u_k}_{p} \geqslant \eta_L$. 
\end{proof}

\subsection{Proof of Theorem~\ref{thm:main}}\label{sec:conv:main}

We start with the following two intermediate results. 
\begin{proposition}\label{lem:if any subsequence is separated from zero}
The sequence
$\ungulata{\alpha_k}$ converges to the unique fixed point $\alpha^*$ of the mapping $\alpha \mapsto \beta(\alpha)$, and
\begin{equation}\label{norms tend globally} 
\Lim{k \to \infty} \norm{u_k}_{p} = \prnth{\frac{\alpha^*}{R^*}}^{\frac{1}{p-1}}, 
\quad 
\Lim{k \to \infty} \norm{\nabla u_k}_{p} = \prnth{\frac{\alpha^*}{\sqrt[p]{R^*}}}^{\frac{1}{p-1}}. 
\end{equation}
Moreover, an arbitrary increasing sequence of natural numbers $\ungulata{k_n}$ can be reduced in such a way that $\ungulata{u_{k_n}}$ converges strongly in $\Wo$ to some higher eigenfunction $u$ corresponding to the eigenvalue $R^*$. 
Furthermore, for any $i \in \mathbb{N}$, the shifted subsequence $\ungulata{u_{k_n + i}}$ also converges to $u$ strongly in $\Wo$.
\end{proposition}
\begin{proof}
Let $\ungulata{k_n}$ be an arbitrary index sequence.
First, note that Lemma~\ref{lem:bounds} allows to reduce $\{k_n\}$ in such a way that $\{u_{k_n}\}$ converges strongly in $L^p(\Omega)$ to some $u_{(0)} \in \Wo$,
and by Lemma~\ref{lem:separatedness} we have $u_{(0)} \ne 0$.
Then, applying Proposition~\ref{lem:conv:sep}, we can thin out $\{k_n\}$ in such a way that the sequence $\{\tilde u_{k_n}\}$ converges to $\tilde u_{(0)} = \|u_{(0)}\|_p^{-1} u_{(0)}$ strongly in $\Wo$,
and the same proposition guarantees that $u_{(0)}$ is an eigenfunction associated with the higher eigenvalue $R^*$.

Recall that $\ungulata{\alpha_{k_n}}, \ungulata{\beta_{k_n}} \subset (0;1)$ and the function $\alpha \mapsto \beta(\alpha)$ is continuous.
Hence, there is a thin-out of $\ungulata{k_n}$ assuring the existence of $\alpha_{(0)} \in [0;1]$ such that 
\begin{equation}
\label{tilde u_{k_n} to tilde u_{(0)}0} 
\Lim{n \to \infty} \alpha_{k_n} = \alpha_{(0)}, \quad \Lim{n \to \infty} \beta_{k_n} = \beta(\alpha_{(0)}).
\end{equation}
Applying Corollary~\ref{strong partial limits of solutions} with 
$f_{k_n} = \alpha_{k_n}^{1/(p-1)}\tilde u_{k_n}^+ - \beta_{k_n}^{1/(p-1)} \tilde u_{k_n}^-$
and noticing that \eqref{tilde u_{k_n} to tilde u_{(0)}0} and the strong convergence of $\{\tilde u_{k_n}\}$ in $\Wo$ yield
$$
f_{k_n} \to 
\alpha_{(0)}^{\frac{1}{p-1}}\tilde u_{(0)}^+ - \beta^{\frac{1}{p-1}}(\alpha_{(0)}) \cdot \tilde u_{(0)}^- 
~\text{ strongly in }~ \Wo, 
$$
we can further thin out $\ungulata{k_n}$ to get $u_{k_n + 1} \to u_{(1)}$ strongly in $W_0^{1,p}(\Omega)$, 
where $u_{(1)}$ solves 
\begin{equation}\label{phi-u1}
-\Delta_p u_{(1)} = \alpha_{(0)} \prnth{\tilde u_{(0)}^+}^{p-1} - \beta(\alpha_{(0)}) \cdot \prnth{\tilde u_{(0)}^-}^{p-1} \text{ in } \Omega.
\end{equation} 
Again, by Lemma~\ref{lem:separatedness}, we have $u_{(1)} \neq 0$, and hence 
the convergence of $\ungulata{u_{k_n + 1}}$ implies that $\tilde u_{k_n + 1} \to \tilde u_{(1)}$ strongly in $\Wo$.
Then it follows from Proposition~\ref{lem:conv:sep} that $ u_{(1)}$ is a higher eigenfunction associated with the eigenvalue $R^*$:
\begin{equation}\label{u1 is EF} 
-\Delta_p u_{(1)} = R^* \big[(u_{(1)}^+)^{p-1} - (u_{(1)}^-)^{p-1} \big] \text{ in } \Omega.
\end{equation}
Combining \eqref{phi-u1} with \eqref{u1 is EF}, we see that
\begin{equation}\label{ux ux ja petux}
\alpha_{(0)}^{\frac{1}{p-1}} \, \tilde u_{(0)}^+ = (R^*)^{\frac{1}{p-1}} u_{(1)}^+, \quad \beta^{\frac{1}{p-1}} \big(\alpha_{(0)}\big)\cdot \tilde u_{(0)}^- = (R^*)^{\frac{1}{p-1}} u_{(1)}^-. 
\end{equation}
Since in \eqref{ux ux ja petux} both $u_{(0)}$ and $u_{(1)}$ are eigenfunctions,
we deduce from \cite[Lemma~2.5 and Remark~2.2]{Drabek on Courant} that $\alpha_{(0)} = \beta(\alpha_{(0)})$,
which means that $\alpha_{(0)}$ is a fixed point of the mapping $\alpha \mapsto \beta(\alpha)$.
Such a fixed point is unique due to the strict decreasing property of $\beta(\cdot)$, and we denote it as $\alpha^*$.
Since, as we have proved, an arbitrarily chosen sequence of indices $\{k_n\}$ can be thinned out so that $\alpha_{k_n} \to \alpha^*$, we conclude that the whole sequence $\{\alpha_k\} \subset (0;1)$ converges to $\alpha^*$; by the continuity of $\beta(\cdot)$, we also have $\beta_k \to \beta(\alpha^*) = \alpha^*$.

We can now rewrite \eqref{ux ux ja petux} as
\begin{equation}\label{wse chastushki my propeli} 
u_{(1)} = \prnth{\frac{\alpha^*}{R^*}}^{\frac{1}{p-1}} \tilde u_{(0)},
\end{equation}
implying
\begin{equation}
	\norm{u_{(1)}}_p = \prnth{\frac{\alpha^*}{R^*}}^{\frac{1}{p-1}}. 
\end{equation}
Since neither $\alpha^*$ nor $R^*$ depends on the choice of $\ungulata{k_n}$, we can apply the above reasoning to get $\norm{u_{k_l + 1}}_{p}^{p-1} \to {\alpha^*}/{R^*}$
for any index sequence $\ungulata{k_l}$ such that  $\{\norm{u_{k_l + 1}}_{p}\}$ converges. 
This fact in combination with Lemma~\ref{lem:bounds} implies that 
\begin{equation}\label{Lp norm for k_i} 
\Lim{k \to \infty} \norm{u_{k}}_{p}^{p-1} = \frac{\alpha^*}{R^*}. 
\end{equation}
Since, in view of \eqref{prop:monotonicity}, we have
$$ 
R^* = \Lim{k \to \infty} \frac{\norm{\nabla u_k}_{p}^p}{\norm{u_k}_{p}^p}, 
$$
we obtain both the convergences in \eqref{norms tend globally}.

From \eqref{Lp norm for k_i}, it follows that $\norm{u_{(0)}}_p = (\alpha^*/R^*)^{1/(p-1)}$, so that \eqref{wse chastushki my propeli} reads as
$$ u_{(1)} = \prnth{\frac{\alpha^*}{R^*}}^{\frac{1}{p-1}} \norm{u_{(0)}}_p^{-1} u_{(0)} = u_{(0)}, 
$$
i.e., the strong limits of $\{u_{k_n}\}$ and $\{u_{k_n + 1}\}$ coincide. 
Note, however, that we employed some additional thin-outs to assure that $u_{k_n + 1} \to u_{(1)} = u_{(0)}$  strongly in $\Wo$. 
It remains to show that once $\ungulata{u_{k_n}}$ converges, no further thin-outs are necessary to guarantee the convergence of $\ungulata{u_{k_n + 1}}$.

Suppose that the subsequence $\ungulata{u_{k_n}}$ converges to $u_{(0)}$ strongly in $\Wo$, while $\ungulata{u_{k_n + 1}}$ does not.
Then there exist $\varepsilon > 0$ and a thin-out of $\ungulata{k_n}$ such that
$$ 
\norm{\nabla (u_{k_n + 1} - u_{(0)})}_{p} > \varepsilon \quad \text{for any } n \in \mathbb{N}.
$$
However, we already established that $\ungulata{k_n}$ can always be further reduced to assure the strong convergence of $\ungulata{u_{k_n + 1}}$, and that the corresponding limit will always coincide with $u_{(0)}$, which contradicts the initial claim. 

Continuing by induction, we conclude that, for any $i \in \mathbb{N}$, the shift 
$\ungulata{u_{k_n + i}}$ must converge to $u_{(0)}$ strongly in $\Wo$ without any further thin-outs.
\end{proof}

\begin{lemma}\label{lem:U}
Let ${\cal U}$ be the collection of all strong limit points of the sequence $\ungulata{u_k}$ in $\Wo$. 
Suppose ${\cal U} \not \subset \overline{B(u, \varepsilon)}$ for some $u \in {\cal U}$ and $\varepsilon > 0$, where $B(u, \varepsilon)$ is the $\varepsilon$-neighborhood of $u$ in $\Wo$.
Then there exists $v \in {\cal U}$ such that $\norm{\nabla (u - v)}_{p} = \varepsilon$.
\end{lemma}
\begin{proof}
Consider the following set of indices: 
$$
I = \ungulata{k \in \mathbb{N}\!:~ \norm{\nabla(u_k - u)}_{p} < \varepsilon ~\text{and}~ \norm{\nabla (u_{k + 1} - u)} \geqslant \varepsilon}.
$$ 
If $I$ is finite, then there exists $N$ such that either $\ungulata{u_k}_{k=N}^{\infty} \subset B(u, \varepsilon)$ or $\ungulata{u_k}_{k=N}^{\infty} \subset W_0^{1,p}(\Omega) \setminus B(u, \varepsilon)$, which is impossible by the assumption on $B(u,\varepsilon)$. 
Hence, $I$ is infinite.

By Proposition~\ref{lem:if any subsequence is separated from zero}, the subsequence $\ungulata{u_{k_n} \mid k_n \in I} \subset B(u, \varepsilon)$ can be reduced to a sub-subsequence converging to some $v \in \overline{B(u,\varepsilon)} \cap {\cal U}$ strongly in $W_0^{1,p}(\Omega)$.  
Then, again by Proposition~\ref{lem:if any subsequence is separated from zero}, the corresponding thin-out of $\ungulata{u_{k_n + 1}} \subset W_0^{1,p}(\Omega) \setminus B(u,\varepsilon)$ also converges to $v$, implying $v \in W_0^{1,p}(\Omega) \setminus B(u, \varepsilon)$. Since $v \in {\cal U}$ lies both in $\overline{B(u,\varepsilon)}$ and $W_0^{1,p}(\Omega) \setminus B(u, \varepsilon)$, we conclude that $v \in \partial B(u, \varepsilon) \cap {\cal U}$.
\end{proof}

We are finally ready to prove Theorem~\ref{thm:main}.

\begin{proof}[Proof of Theorem~\ref{thm:main}]
Statements \ref{thm:main:R is an eigenvalue}, \ref{thm:main:partial limits are eigenfunctions}, \ref{thm:main:shifts}, and \ref{thm:main:alphas} of Theorem~\ref{thm:main} are given  by Propositions~\ref{lem:conv:sep} and \ref{lem:if any subsequence is separated from zero}. Statement \ref{thm:main:dist to 0} follows immediately from \ref{thm:main:shifts}.

In order to prove  Statement~\ref{thm:main:structure of U}, suppose ${\cal U}$ has an isolated point $u$, whereas ${\cal U} \setminus \ungulata{u}$ is non-empty, i.e., there exists $\varepsilon > 0$ such that $\overline{B(u,\varepsilon)} \cap {\cal U} = \ungulata{u}$ and ${\cal U} \not \subset \overline{B(u, \varepsilon)}$. 
Then it follows from Lemma~\ref{lem:U} that $\partial B(u, \varepsilon) \cap {\cal U} \ne \varnothing$, which contradicts the choice of $\varepsilon$. 
Thus,  ${\cal U}$  either consists of the only point $u$ or has no isolated elements at all.

Recall that by Proposition~\ref{lem:if any subsequence is separated from zero} the $L^p(\Omega)$-norms of the elements in ${\cal U}$ coincide:
$$ \norm{u}_p = a \quad \text{for any } u \in {\cal U}, \quad \text{where } a := \prnth{\frac{\alpha^*}{R^*}}^\frac{1}{p-1}. $$
If the eigenvalue $R^*$ is simple, then by \ref{thm:main:partial limits are eigenfunctions} all the elements in ${\cal U}$ are multiples of a single eigenfunction $\varphi$, and we can let $\norm{\varphi}_p = a$, so that ${\cal U} \subset \{\pm \varphi\}$. However, assuming ${\cal U} = \{\pm \varphi\}$ would mean that the set ${\cal U}$ consists of two isolated points, in contradiction to the already obtained result.
Therefore, ${\cal U}$ has to be a singleton, which finishes the proof of Statement~\ref{thm:main:structure of U}. 
\end{proof}

\medskip
\begin{remark}\label{remark:counterexample}
It is tempting to speculate whether the ``shift property'' given by Statement~\ref{thm:main:shifts} of Theorem~\ref{thm:main} implies the convergence of the whole sequence $\ungulata{u_k}$ generated by \hyperref[alg]{Algorithm A}. However, one can easily provide an example of a complete metric space $X$ and a sequence $\ungulata{x_k} \subset X$ such that
\begin{enumerate}[label={\rm(S\arabic*)}]
\item\label{shift property:1} any subsequence of $\ungulata{x_k}$ contains a convergent sub-subsequence,
\item\label{shift property:2} any shift $\ungulata{x_{k_n + i}}$, $i \in \mathbb{N}$, of any convergent subsequence $\ungulata{x_{k_n}}$ converges to the same limit as $\ungulata{x_{k_n}}$, \item\label{shift property:3} the whole sequence $\{x_k\}$ does not converge.
\end{enumerate}
For instance, consider the sequence of reals $\{x_k\}$ such that $x_0=0$ and 
\begin{equation}\label{countersequence}
x_{k+1} = x_k + \frac{\sigma_{k+1}}{k+1}
~\text{ for } k \geqslant 0, 
\end{equation} 
where $\sigma_0 = 1$ and for $k \geqslant 1$ the coefficient $\sigma_{k+1}$ is given by
$$ 
\sigma_{k+1} = 
\begin{cases}
1, &\text{ if } x_k < 0, \\
-1, &\text{ if } x_k > 1, \\
\sigma_k, &\text{ if } 0 \leqslant x_k \leqslant 1,
\end{cases}
$$
see Fig.~\ref{fig:enter-label}. 
Since the harmonic series diverges, the sequence $\{x_k\}$ passes through the points 0 and 1 infinitely many times, which yields \ref{shift property:3}. 
However, $\{x_k\}$ is obviously bounded and therefore contains a convergent subsequence $x_{k_n} \to \xi \in \mathbb{R}$. 
Note that
$$ x_{k_n + 1} = x_{k_n} + \frac{\sigma_{k_n + 1}}{k_n + 1} \to \xi + 0 = \xi, $$
so that the properties \ref{shift property:1} and \ref{shift property:2} are also satisfied. 
The existence of such a counterexample indicates that the convergence of the whole sequence $\ungulata{u_k}$ generated by \hyperref[alg]{Algorithm A} can not be derived from the results \ref{thm:main:R is an eigenvalue}-\ref{thm:main:structure of U} alone without some additional properties of \hyperref[alg]{Algorithm A} taken into consideration.

\begin{figure}[h]
\centering
\includegraphics[scale=0.4]{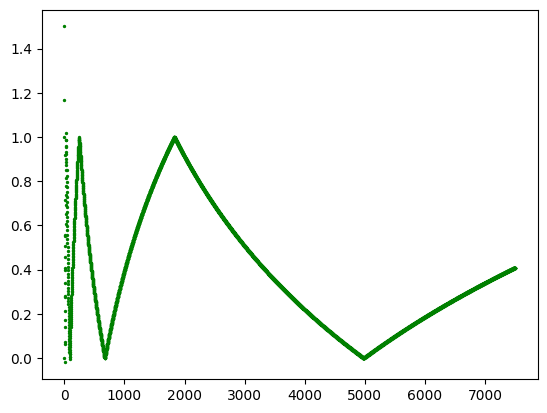}
\caption{The first 7501 elements of the sequence \eqref{countersequence}.}
\label{fig:enter-label}
\end{figure}
\end{remark}

\section{Numerical results}\label{sec:num}

In this section, we demonstrate the performance of \hyperref[alg]{Algorithm A} in a series of numerical experiments. In order to implement \hyperref[alg]{Algorithm A}, we used the \verb|FEniCS| finite-element library for \verb|python|. The boundary value problems \eqref{phi} were solved using the \verb|NonLinearVariationalSolver| tool provided by \verb|FEniCS|, which applies Newton iterations to non-linear systems approximating variational equations. We note that the use of Newton iterations restricts the convergence of our implementation, especially for very small and very large values of $p$. To solve \eqref{eq:R+=R-}, we used the \verb|ridder| root-finding  algorithm provided by the \verb|scipy.optimize| package. The full codes of \verb|python| programs corresponding to the numerical experiments described in this section are available at the git repository \cite{git}.

Hereinafter, we let $\Omega$ be the unit square $(0;1)^2$. 
In the case $p=2$, the spectrum of $\Omega$ is exhausted by the eigenvalues given by the formula $\pi^2 (m^2 + n^2)$, $m, n \in \mathbb{N}$. 
In particular, we get
\begin{alignat*}{2}
&\lambda_2 (\Omega, 2) =\lambda_3 (\Omega, 2) = 5\pi^2 \approx 49.35 
&&\quad \text{for}~ \min\{m, n\} = 1,~ \max\{m,n\} = 2,\\
&\lambda_4 (\Omega, 2) = 8\pi^2 \approx 78.96
&&\quad \text{for}~ m = n = 2,\\
&\lambda_5 (\Omega, 2) = \lambda_6 (\Omega, 2) = 10 \pi^2 \approx 98.70
&&\quad \text{for}~ \min\{m, n\} = 1,~ \max\{m,n\} = 3.
\end{alignat*}
For $p \ne 2$, no analytic expressions for the eigenvalues, including $\lambda_2 (\Omega, p)$, are available. 
To justify our computations, we compare the resulting estimates with the approximations obtained in \cite{Horak} for the square $2\Omega = (0;2)^2$, rescaling these by the factor of $2^p$. 

Besides $\lambda_2 (\Omega, p)$ itself, the work \cite{Horak} also investigates the values $\lambda_{{\cal S}_1} (\Omega, p)$ and $\lambda_{{\cal S}_2} (\Omega, p)$, where ${\cal S}_1$ is the class of functions odd about one of the square's \textit{middle lines} and even about the other,
${\cal S}_2$ is the class of functions odd about one of the square's \textit{diagonals} and even about the other, and $\lambda_{{\cal S}_i} (\Omega, p)$ is the least of all eigenvalues having some eigenfunction $u \in {\cal S}_i$ associated with them, $i = \overline{1;2}$. 
It is conjectured in \cite[pp.~477-478]{diening} that the second eigenfunction belongs to ${\cal S}_1$ for $p < 2$, and to ${\cal S}_2$ for $p > 2$; for $p = 2$, the second eigenvalue is associated with both ${\cal S}_1$- and ${\cal S}_2$-eigenfunctions. 
We let $\lambda_2^{\mathrm{H}}$, $\lambda_{{\cal S}_1}^{\mathrm{H}}$, and $\lambda_{{\cal S}_2}^{\mathrm{H}}$ denote the corresponding values from \cite{Horak} rescaled by the factor of $2^p$, so that
\begin{align}
\label{square lampdas:p<2}\lambda_2^{\mathrm{H}}=\lambda_{{\cal S}_1}^{\mathrm{H}}<\lambda_{{\cal S}_2}^{\mathrm{H}} & \quad \text{for }p<2 ,\\
\label{square lampdas:p=2}\lambda_2^{\mathrm{H}}=\lambda_{{\cal S}_1}^{\mathrm{H}}=\lambda_{{\cal S}_2}^{\mathrm{H}} & \quad \text{for } p=2,\\
\label{square lampdas:p>2}\lambda_2^{\mathrm{H}}=\lambda_{{\cal S}_2}^{\mathrm{H}}<\lambda_{{\cal S}_1}^{\mathrm{H}} & \quad \text{for } p>2.
\end{align}

In the following three examples, we test \hyperref[alg]{Algorithm A} by choosing different types of an initial guess $u_0$.

\begin{example}\label{ex:square mid} First consider the following initial guess:
\begin{equation}\label{square mid}
u_0^{{\cal S}_1} = xy \,(x-1)(y-1)\prnth{x-\frac{1}{2}}.
\end{equation}
Obviously, $u_0 \in {\cal S}_1$, so \hyperref[alg]{Algorithm A} has to preserve the evenness of $u_0$ at each iteration, see Remark~\ref{rem:initial-guess}/\ref{rem:initial-guess:2}. 
However, it might hypothetically happen that the oddness will be lost, since the right-hand side of \eqref{phi} lacks the antisymmetry for $\alpha \neq \beta(\alpha)$. 
Nevertheless, it is natural to expect the convergence to $\lambda_{{\cal S}_1}(\Omega,p)$. 
The approximate values $R[u_5] \approx \lambda_{{\cal S}_1}(\Omega,p)$ produced by 5 iterations of \hyperref[alg]{Algorithm A} for various $p$, as well as their comparison to the estimates from \cite{Horak}, are given in Table~\ref{table:square mid}. 
We see that our results are in agreement with those of \cite{Horak}. In particular, the values of $R[u_5]$ obtained for $p > 2$ lie closer to $\lambda_{{\cal S}_1}^{\mathrm{H}}$ than to $\lambda_2^{\mathrm{H}}$, reflecting the symmetry of $u_0^{{\cal S}_1}$ (recall \eqref{square lampdas:p>2}).
Detailed graphs, including the behavior of the Rayleigh quotients and the norms of the differences $|u_{k+1} - u_k|$, as well as the resulting approximations graphs, are given on Fig.~\ref{fig:square mid detailed}. Note that the Rayleigh quotients are monotonically decreasing, just as predicted by Theorem~\ref{th:existence-of-alphaK}.

\begin{table}%[h!]
\centering
\begin{subtable}[t]{0.45\textwidth}
\begin{tabular}[t]{|c|c|c|c|c|c|}
	\hline
	$p$ & $R[u_5]$ & $\frac{\lambda_2^{\mathrm{H}} - R[u_5]}{\lambda_2^{\mathrm{H}}}$ & $\frac{\lambda_{{\cal S}_1}^{\mathrm{H}} - R[u_5]}{\lambda_{{\cal S}_1}^{\mathrm{H}}}$\\ \hline
	1.6 & 23.68 & -0.00177 & idem. \\ \hline
	1.7 & 28.61 & -0.00180 & idem. \\ \hline
	1.8 & 34.44 & -0.00187 & idem. \\ \hline
	1.9 & 41.32 & -0.00199 & idem. \\ \hline
	2.0 & 49.46 & -0.00209 & idem. \\ \hline
	2.1 & 59.06 & -0.00668 & -0.00228 \\ \hline
	2.2 & 70.38 & -0.01151 & -0.00237 \\ \hline
	2.3 & 83.74 & -0.01672 & -0.00258 \\ \hline
	2.4 & 99.48 & -0.02223 & -0.00271 \\ \hline
	2.5 & 118.02 & -0.02806 & -0.00290 \\ \hline
	2.6 & 139.83 & -- & -- \\ \hline
	2.7 & 165.49 & -- & -- \\ \hline
	2.8 & 195.66 & -- & -- \\ \hline
	2.9 & 231.11 & -- & -- \\ \hline
	3.0 & 272.74 & -0.06182 & -0.00401 \\ \hline
	3.1 & 321.60 & -- & -- \\ \hline
	3.2 & 378.94 & -- & -- \\ \hline
\end{tabular}
\end{subtable}%
\begin{subtable}[t]{0.45\textwidth}
\begin{tabular}[t]{|c|c|c|c|c|c|}
	\hline
	$p$ & $R[u_5]$ & $\frac{\lambda_2^{\mathrm{H}} - R[u_5]}{\lambda_2^{\mathrm{H}}}$ & $\frac{\lambda_{{\cal S}_1}^{\mathrm{H}} - R[u_5]}{\lambda_{{\cal S}_1}^{\mathrm{H}}}$\\ \hline
	3.3 & 446.19 & -- & -- \\ \hline
	3.4 & 525.02 & -- & -- \\ \hline
	3.5 & 617.40 & -- & -- \\ \hline
	3.6 & 725.61& -- & -- \\ \hline
	3.7 & 852.32 & -- & -- \\ \hline
	3.8 & 1000.63 & -- & -- \\ \hline
	3.9 & 1174.18 & -- & -- \\ \hline
	4.0 & 1377.19 & -0.15139 & -0.00734 \\ \hline
	4.1 & 1614.59 & -- & -- \\ \hline
	4.2 & 1892.11 & -- & -- \\ \hline
	4.3 & 2216.45 & -- & -- \\ \hline
	4.4 & 2595.40 & -- & -- \\ \hline
	4.5 & 3038.03 & -- & -- \\ \hline
	4.6 & 3554.91 & -- & -- \\ \hline
	4.7 & 4158.34 & -- & -- \\ \hline
	4.8 & 4862.65 & -- & -- \\ \hline
	4.9 & 5684.50 & -- & -- \\ \hline
	5.0 & 6643.28 & -0.26904 & -0.01230 \\ \hline
\end{tabular}
\end{subtable}
\caption{Example~\ref{ex:square mid}: the approximate eigenvalues of the unit square obtained by applying 5 iterations of \hyperref[alg]{Algorithm A} with the initial guess \eqref{square mid}, and their comparison to the results of \cite{Horak}.}
\label{table:square mid}
\end{table}
\end{example}

\begin{example}\label{ex:square diag}
Now let the initial guess be defined as
\begin{equation}\label{square diag} 
	u_0^{{\cal S}_2} = xy \, (x - 1) 
	(y - 1) (x - y). 
\end{equation}
This function belongs to ${\cal S}_2$. As in Example~\ref{ex:square mid}, we summarize the obtained numerical results in Table~\ref{table:square diag} and on Fig.~\ref{fig:square diag detailed}, this time comparing $R[u_5]$ to $\lambda_{{\cal S}_2}^{\mathrm{H}}$ instead of $\lambda_{{\cal S}_1}^{\mathrm{H}}$. The values of the corresponding relative errors indicate that our computations comply with those of \cite{Horak}. Note that for $p < 2$ the values of $R[u_5]$ tend to be very close to $\lambda_{{\cal S}_2}^{\mathrm{H}}$ and significantly greater than $\lambda_2^{\mathrm{H}}$, matching the symmetry of $u_0^{{\cal S}_2}$ (recall \eqref{square lampdas:p<2}). As before, we observe the decrease of the Rayleigh quotients predicted by Theorem~\ref{th:existence-of-alphaK}.

\begin{table}%[h!]
\centering
\begin{subtable}[t]{0.4\textwidth}
\begin{tabular}[t]{|c|c|c|c|c|c|}
	\hline
	$p$ & $R[u_5]$ & $\frac{\lambda_2^{\mathrm{H}} - R[u_5]}{\lambda_2^{\mathrm{H}}}$ & $\frac{\lambda_{{\cal S}_2}^{\mathrm{H}} - R[u_5]}{\lambda_{{\cal S}_2}^{\mathrm{H}}}$\\ \hline
	1.6 & 24.02 & -0.01603 & -0.00130 \\ \hline
	1.7 & 28.92 & -0.01281 & -0.00135 \\ \hline
	1.8 & 34.69 & -0.00934 & -0.00141 \\ \hline
	1.9 & 41.47 & -0.00563 & -0.00155 \\ \hline
	2.0 & 49.43 & -0.00162 & -0.00162  \\\hline
\end{tabular}
\end{subtable}%
\begin{subtable}[t]{0.3\textwidth}
\begin{tabular}[t]{|c|c|c|}
	\hline
	$p$ & $R[u_5]$ & $\frac{\lambda_2^{\mathrm{H}} - R[u_5]}{\lambda_2^{\mathrm{H}}}$ \\ \hline
	
	2.1 & 58.77 & -0.00179  \\ \hline
	2.2 & 69.71 & -0.00188  \\ \hline
	2.3 & 82.53 & -0.00205 \\ \hline
	2.4 & 97.53 & -0.0022  \\ \hline
	2.5 & 115.07 & -0.00236 \\ \hline
	2.6 & 135.55 & -- \\ \hline
	2.7 & 159.46 & -- \\ \hline
	2.8 & 187.35 & -- \\ \hline
	2.9 & 219.85 & -- \\ \hline
	3.0 & 257.71 & -0.00333 \\ \hline
	3.1 & 301.78 & -- \\ \hline
	3.2 & 353.04 & -- \\ \hline
	3.3 & 412.64 & -- \\ \hline
	3.4 & 481.89 & -- \\ \hline
	3.5 & 562.3 & -- \\ \hline
\end{tabular}
\end{subtable}%
\begin{subtable}[t]{0.3\textwidth}
\begin{tabular}[t]{|c|c|c|}
	\hline
	$p$ & $R[u_5]$ & $\frac{\lambda_2^{\mathrm{H}} - R[u_5]}{\lambda_2^{\mathrm{H}}}$ \\ \hline
	3.6 & 655.63 & -- \\ \hline
	3.7 & 763.9 & -- \\ \hline
	3.8 & 889.45 & -- \\ \hline
	3.9 & 1034.95 & -- \\ \hline
	4.0 & 1203.50 & -0.00618 \\ \hline
	4.1 & 1398.68 & --\\ \hline
	4.2 & 1624.60 & -- \\ \hline
	4.3 & 1885.99 & --  \\ \hline
	4.4 & 2188.32 & -- \\ \hline
	4.5 & 2537.85 & -- \\ \hline
	4.6 & 2941.83 & -- \\ \hline
	4.7 & 3408.57 & -- \\ \hline
	4.8 & 3947.64 & -- \\ \hline
	4.9 & 4570.06 & -- \\ \hline
	5.0 & 5288.49 & -0.01024 \\ \hline
\end{tabular}
\end{subtable}
\caption{Example~\ref{ex:square diag}: the approximate eigenvalues of the unit square obtained by applying 5 iterations of \hyperref[alg]{Algorithm A} with the initial guess \eqref{square diag}, and their comparison with the results of \cite{Horak}.}
\label{table:square diag}
\end{table}
\end{example}

\begin{example}\label{ex:square rad}
Next consider the initial guess $u_0^{{\cal O}}$ which is centrally symmetric around the middle point $M = (1/2; 1/2)$ and whose nodal set is a circle centered at $M$ (see Fig.~\ref{fig:u_0_rad}):
\begin{equation}\label{square rad} 
u_0^{{\cal O}} = xy \, (x - 1)  (y - 1)  \quadr{\frac{1}{4^2} - \prnth{x - \frac{1}{2}}^2 - \prnth{y - \frac{1}{2}}^2}.
\end{equation}
Note that the function $u_0^{{\cal O}}$ is also symmetric with respect to both middle lines and both diagonals of the unit square $\Omega$. 

\begin{figure}[h!]
\centering
\includegraphics[trim=1.9cm 0.25cm 2.5cm 1.5cm, clip,width=0.45\linewidth]{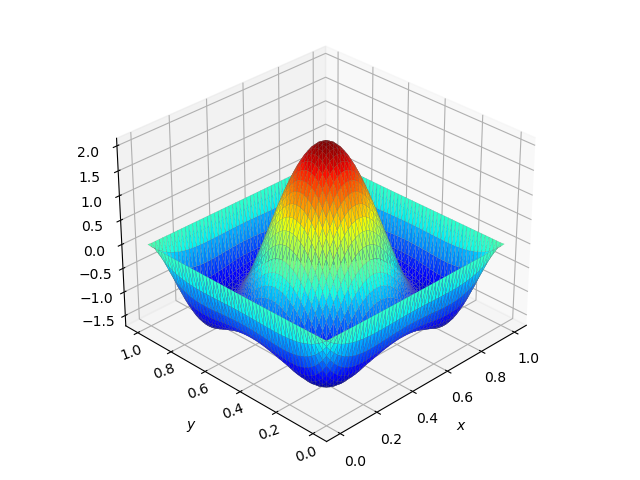}
\includegraphics[width=0.45\linewidth]{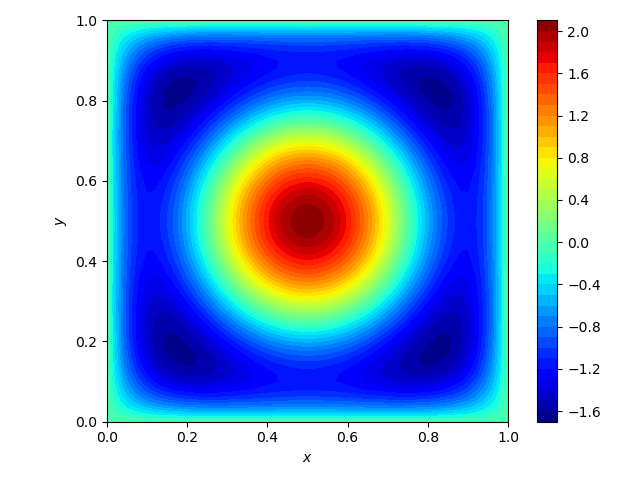}
\caption{Example~\ref{ex:square rad}: the graph of the initial guess \eqref{square rad}.}
\label{fig:u_0_rad}
\end{figure}

Since the solution of the $p$-Poisson problem \eqref{phi} is unique, each function $u_k$ produced by \hyperref[alg]{Algorithm A} has to obey the same symmetries as $u_0^{{\cal O}}$, and hence so do their limits (see Remark~\ref{rem:initial-guess}/\ref{rem:initial-guess:2}).
However, in our numerical experiments, $u_k$ does not inherit the symmetries precisely, due to the mesh geometry and numerical errors. 
We present the corresponding results for $p \in \ungulata{1.7, 2, 3}$ on Fig.~\ref{fig:square rad}.
For each $p$, just as expected, the Rayleigh quotients graphs exhibit monotonically decreasing behavior.
However, the decrease is not ``uniform'', as we observe three visually distinct stabilization stages on each graph.
The switch between the stabilization stages occurs rapidly, corresponding to the gradual increase of the norms $\norm{u_{k+1} - u_k}_p$.
This increase reflects the process of $u_k$ changing its graph's geometry, in particular its symmetry patterns. 
Namely, the first stabilization stage produces a function preserving all the symmetries of $u_0^{\cal O}$, 
the second one produces a centrally symmetric function having the diagonal symmetry without the middle-line symmetry, and the third one produces a function with the central antisymmetry. 
In the linear case $p = 2$, such portraits \textit{roughly} correspond to
\begin{alignat*}{2}
&\text{the fifth eigenfunction}~ &&\sin (3\pi x) \sin(\pi y) + \sin (\pi x) \sin(3\pi y),\\
&\text{the fourth eigenfunction}~ &&\sin (2\pi x) \sin(2\pi y),\\
&\text{the second eigenfunction}~ &&\sin (2\pi x) \sin(\pi y) + \alpha \sin (\pi x) \sin(2\pi y)
~~\text{for some}~~ \alpha \in \mathbb{R},
\end{alignat*}
respectively. It is therefore not surprising that the values of $R[u_k]$ at the stabilization stages of $\ungulata{R[u_k]}$ lie close to $\lambda_5 (\Omega, 2)$, $\lambda_4 (\Omega, 2)$, and $\lambda_2 (\Omega, 2)$, in the order of occurence:
$$ 
\frac{R[u_8] - \lambda_5 (\Omega, 2)}{\lambda_5 (\Omega, 2)} \approx 0.10\%, \quad \frac{R[u_{37}] - \lambda_4 (\Omega, 2)}{\lambda_4 (\Omega, 2)} \approx -0.03\%, \quad \frac{R[u_{100}] - \lambda_2 (\Omega, 2)}{\lambda_2 (\Omega, 2)} \approx 0.06\%.
$$
A natural explanation to such an ``instability'' could be as follows. Initially, the algorithm was about to converge to an eigenfunction with a central symmetry (the fifth eigenfunction in the case $p = 2$). However, as the numerical errors accumulate at each iteration, the symmetry of the consecutive solutions $u_k$ starts deteriorating,
leading to its total rearrangement.
This in turn allows the algorithm to minimize the Rayleigh quotient in a drastic manner
(down towards the fourth eigenvalue in the case $p = 2$).
Then the same process repeats,
lowering $R[u_k]$ towards $\lambda_2 (\Omega, p)$ this time.
Since the second eigenvalue is the smallest which \hyperref[alg]{Algorithm A} can possibly reach,
the iterative scheme finally achieves full stabilization,
and no further jump is expected.
We assume that using more advanced techniques for solving \eqref{phi} numerically,
i.e., those preserving all the right-hand side's symmetries, 
might prevent such an ``instability''.
\end{example}

\begin{figure}
\centering
\begin{subfigure}{0.33\textwidth}
\includegraphics[width=\linewidth]{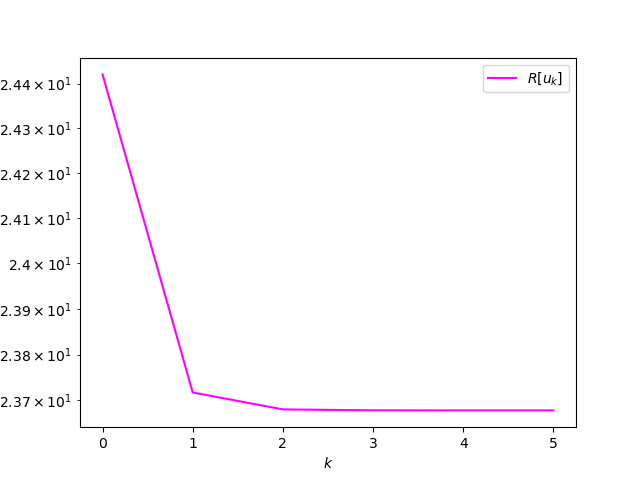}
\end{subfigure}%
\hspace*{\fill}
\begin{subfigure}{0.33\textwidth}
\includegraphics[width=\linewidth]{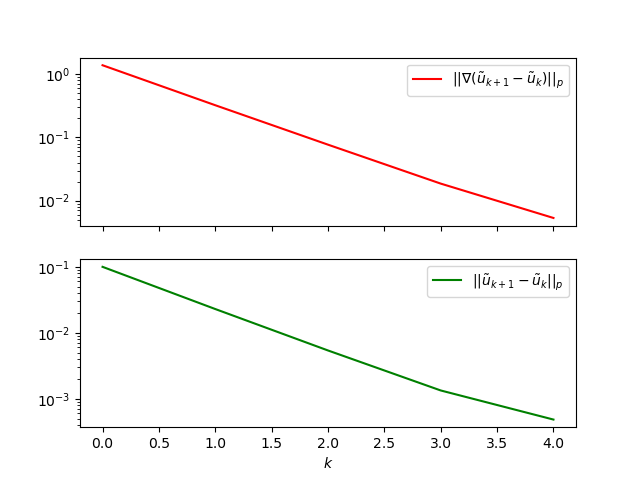}
\end{subfigure}%
\hspace*{\fill}  
\begin{subfigure}{0.33\textwidth}
\includegraphics[trim=1.9cm 0.5cm 2.5cm 1.5cm, clip, width=\linewidth]{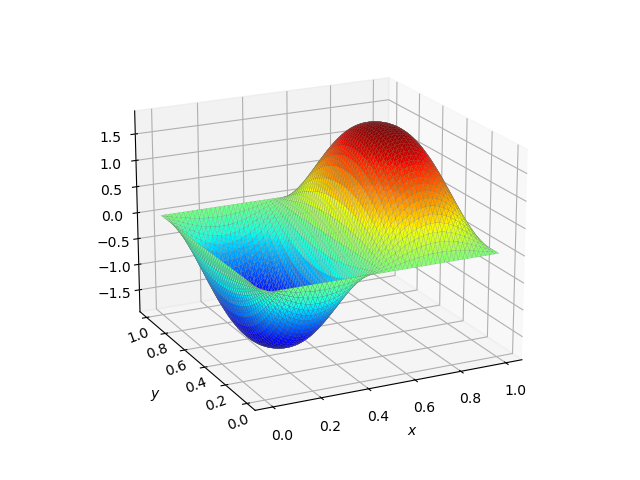}
\end{subfigure}\\[0.25cm]
\begin{subfigure}{0.33\textwidth}
\includegraphics[width=\linewidth]{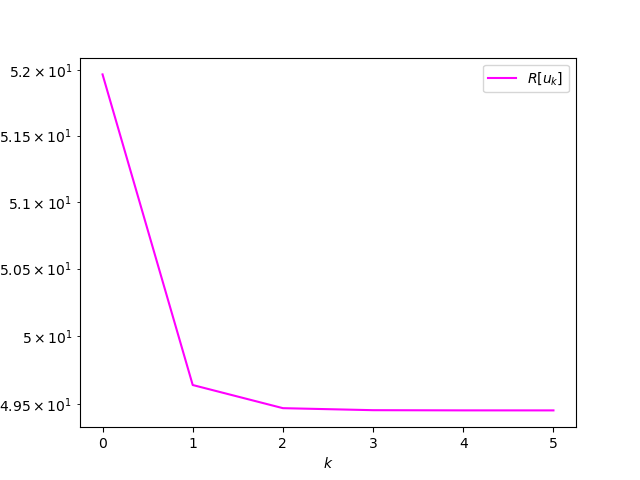}
\end{subfigure}%
\hspace*{\fill} 
\begin{subfigure}{0.33\textwidth}
\includegraphics[width=\linewidth]{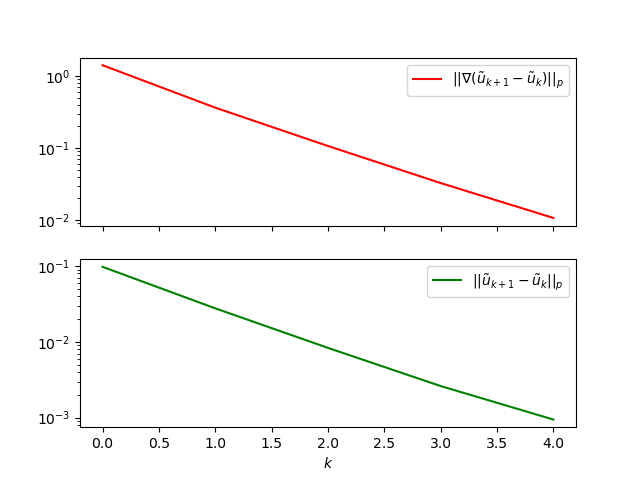}
\end{subfigure}%
\hspace*{\fill} 
\begin{subfigure}{0.33\textwidth}
\includegraphics[trim=1.9cm 0.5cm 2.5cm 1.5cm, clip,width=\linewidth]{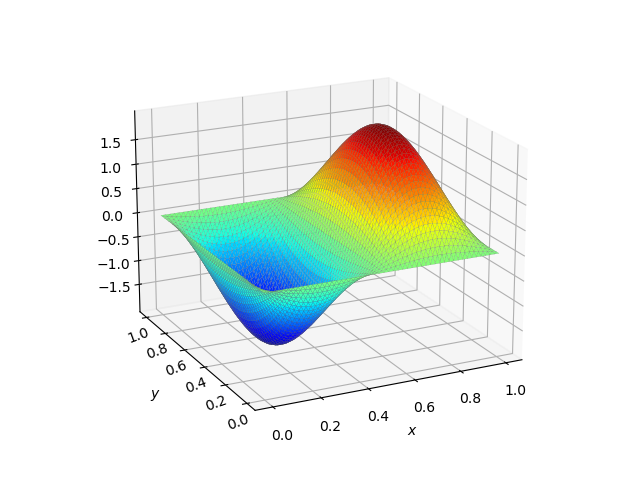}
\end{subfigure}\\[0.25cm]
\begin{subfigure}{0.33\textwidth}
\includegraphics[width=\linewidth]{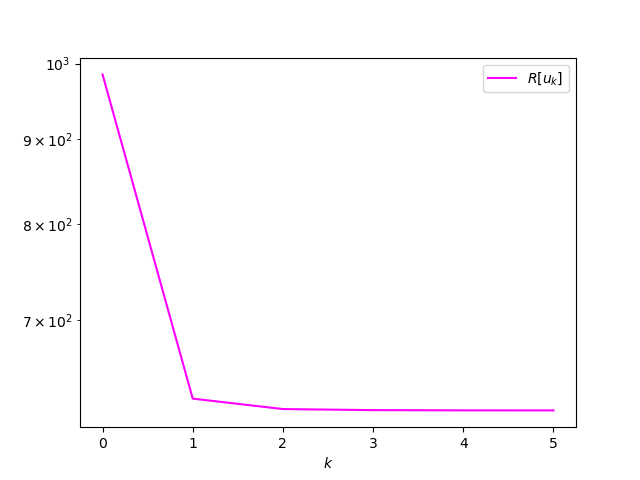}
\end{subfigure}%
\hspace*{\fill} 
\begin{subfigure}{0.33\textwidth}
\includegraphics[width=\linewidth]{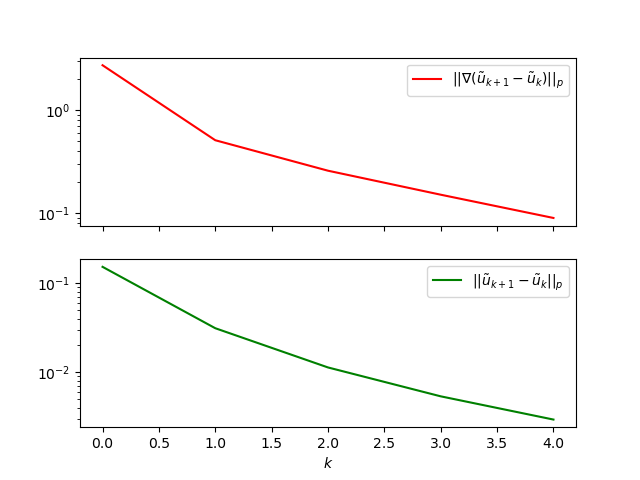}
\end{subfigure}%
\hspace*{\fill} 
\begin{subfigure}{0.33\textwidth}
\includegraphics[trim=1.9cm 0.5cm 2.5cm 1.5cm, clip,width=\linewidth]{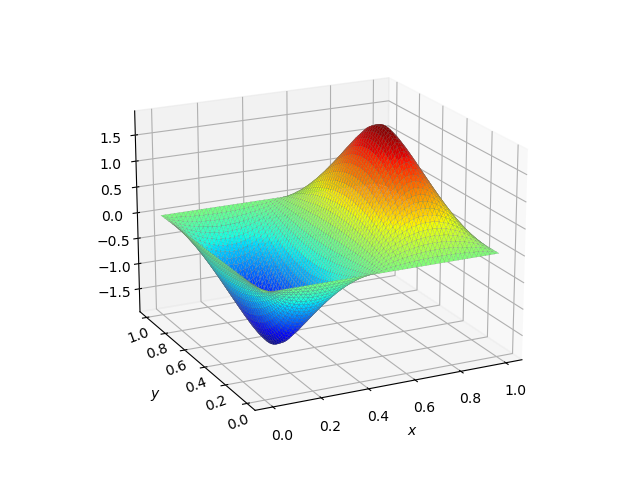}
\end{subfigure}\\[0.25cm]
\begin{subfigure}{0.33\textwidth}
\includegraphics[width=\linewidth]{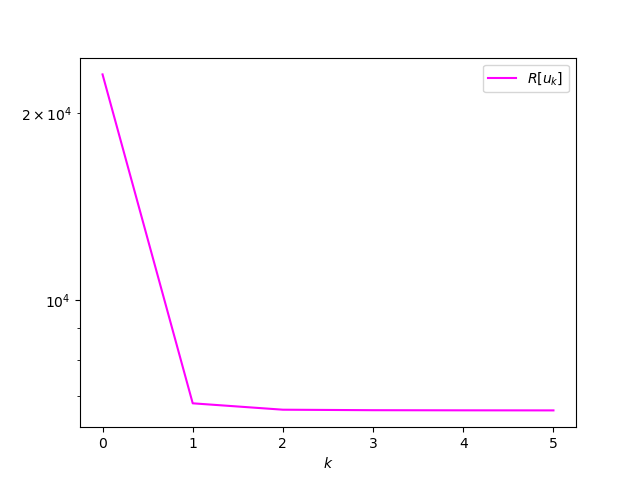}
\end{subfigure}%
\hspace*{\fill}
\begin{subfigure}{0.33\textwidth}
\includegraphics[width=\linewidth]{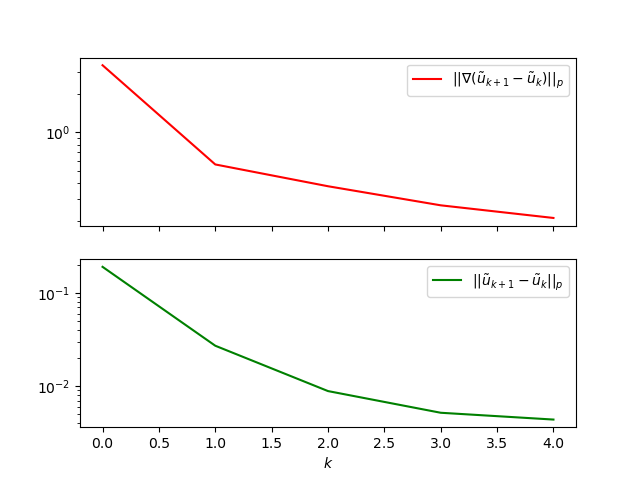}
\end{subfigure}%
\hspace*{\fill}
\begin{subfigure}{0.33\textwidth}
\includegraphics[trim=1.9cm 0.5cm 2.5cm 1.5cm, clip,width=\linewidth]{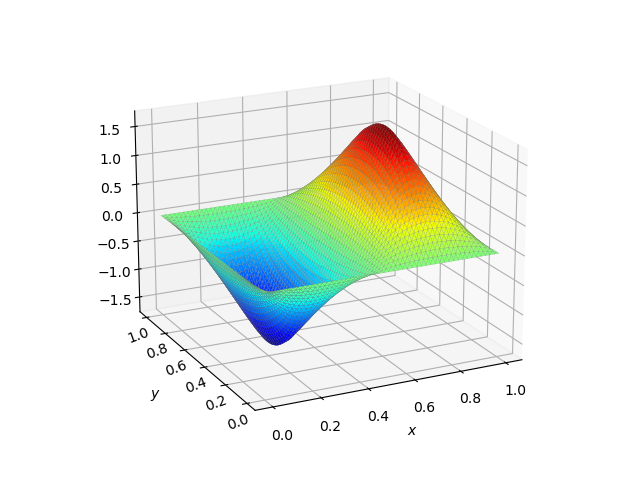}
\end{subfigure}%
\caption{Example~\ref{ex:square mid}: the detailed graphs for some particular values of $p$. Left column -- the Rayleigh quotients sequences $\ungulata{R[u_k]}_{k=0}^5$ (logarithmic $y$ scale); middle column -- the $\Wo$- and $L^p(\Omega)$-norms (red and green, respectively) of the difference between two consecutive approximations $u_k$ and $u_{k+1}$ (logarithmic $y$ scale); right column -- the graphs of the resulting approximations $u_5$. The rows correspond to (from the top to the bottom) $p = 1.6$, $p = 2$, $p = 3.5$, and $p = 5$.}
\label{fig:square mid detailed}
\end{figure}

\begin{figure}
\centering
\begin{subfigure}{0.33\textwidth}
\includegraphics[width=\linewidth]{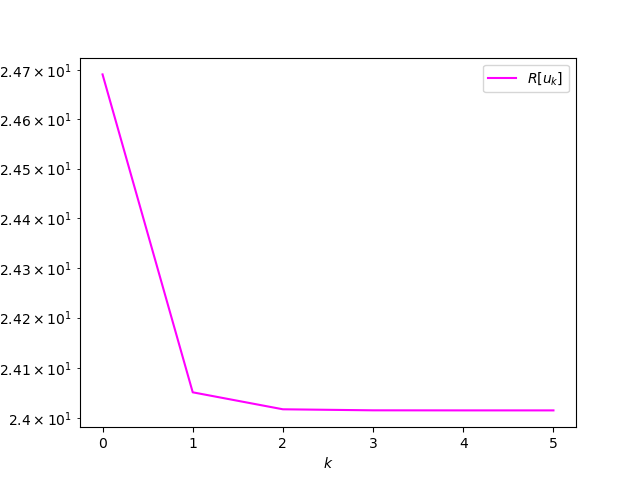}
\end{subfigure}%
\hspace*{\fill} 
\begin{subfigure}{0.33\textwidth}
\includegraphics[width=\linewidth]{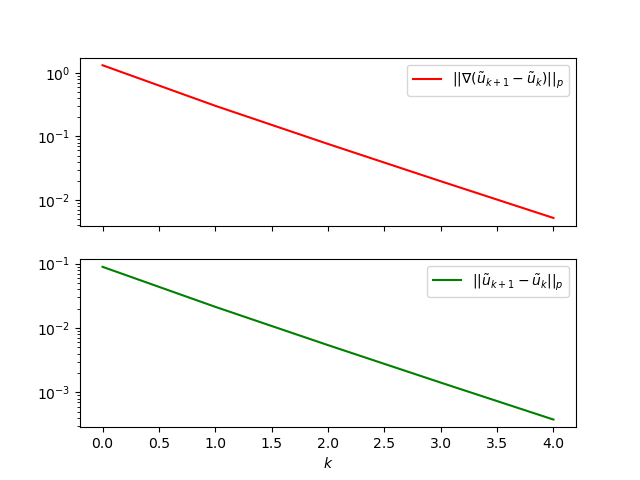}
\end{subfigure}%
\hspace*{\fill}  
\begin{subfigure}{0.33\textwidth}
\includegraphics[trim=1.9cm 0.5cm 2.5cm 1.5cm, clip, width=\linewidth]{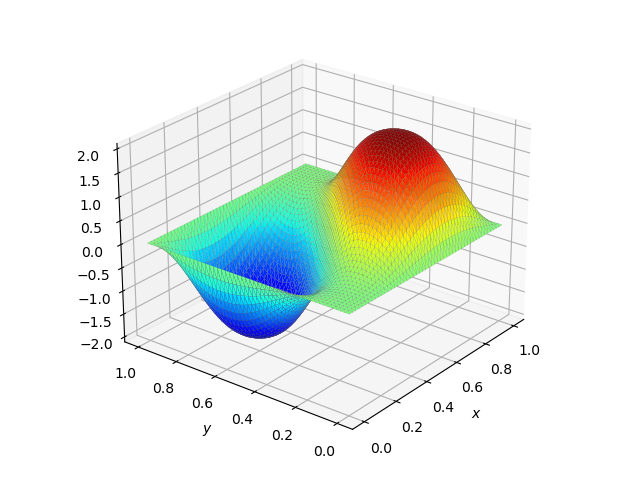}
\end{subfigure}\\[0.25cm]
\begin{subfigure}{0.33\textwidth}
\includegraphics[width=\linewidth]{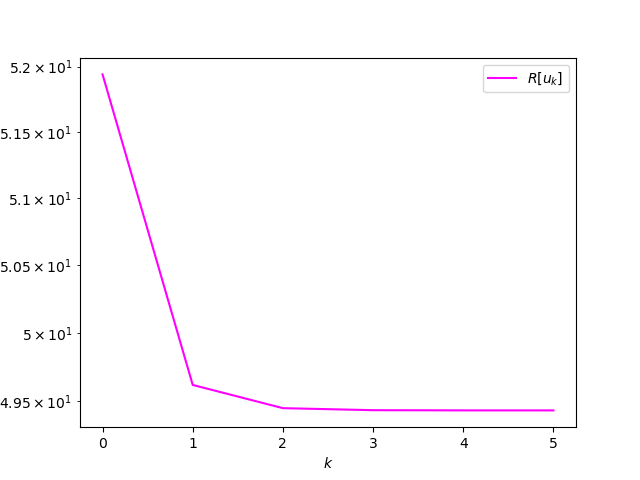}
\end{subfigure}%
\hspace*{\fill} 
\begin{subfigure}{0.33\textwidth}
\includegraphics[width=\linewidth]{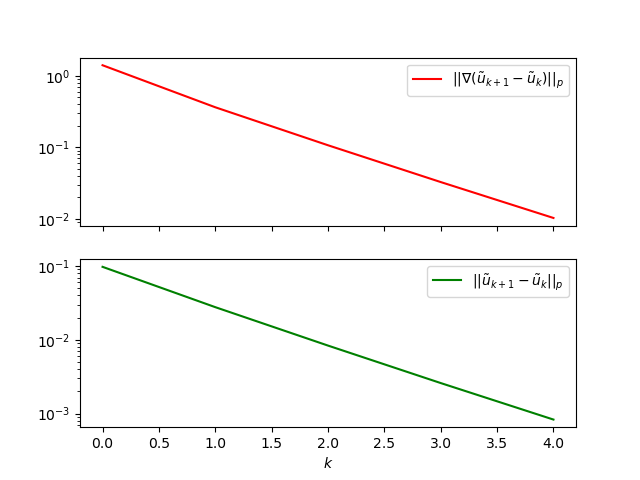}
\end{subfigure}%
\hspace*{\fill} 
\begin{subfigure}{0.33\textwidth}
\includegraphics[trim=1.9cm 0.5cm 2.5cm 1.5cm, clip,width=\linewidth]{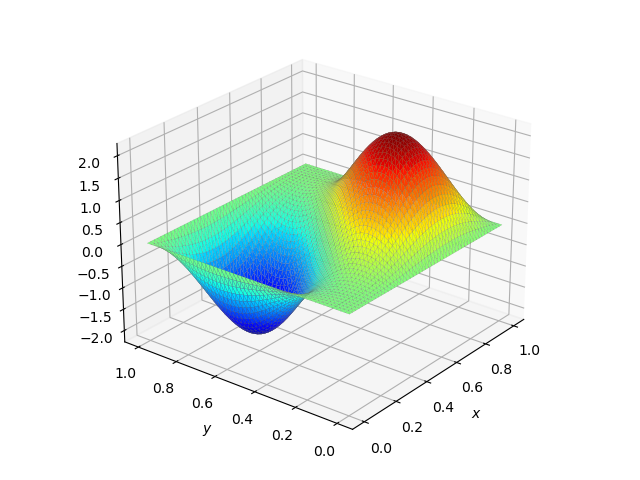}
\end{subfigure}\\[0.25cm]
\begin{subfigure}{0.33\textwidth}
\includegraphics[width=\linewidth]{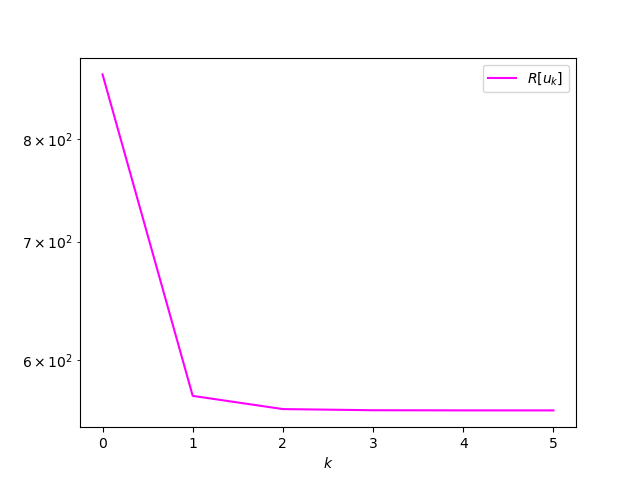}
\end{subfigure}%
\hspace*{\fill} 
\begin{subfigure}{0.33\textwidth}
\includegraphics[width=\linewidth]{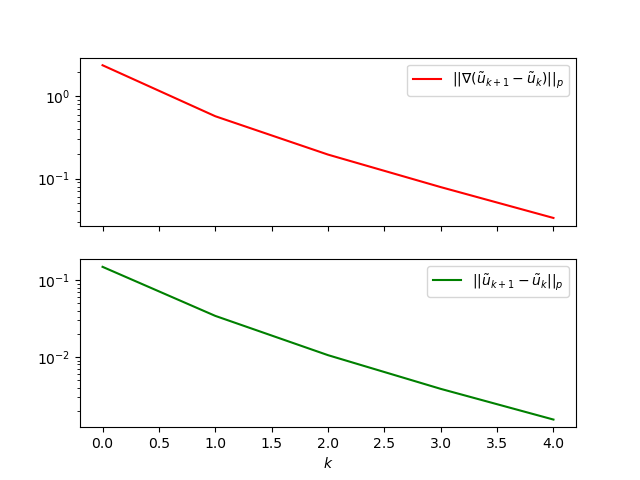}
\end{subfigure}%
\hspace*{\fill} 
\begin{subfigure}{0.33\textwidth}
\includegraphics[trim=1.9cm 0.5cm 2.5cm 1.5cm, clip,width=\linewidth]{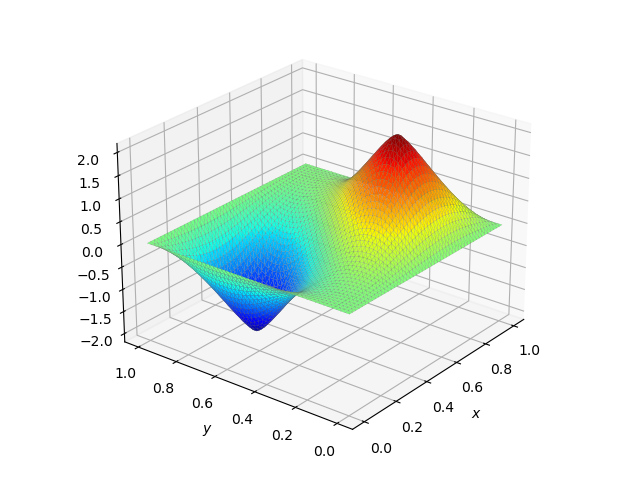}
\end{subfigure}\\[0.25cm]
\begin{subfigure}{0.33\textwidth}
\includegraphics[width=\linewidth]{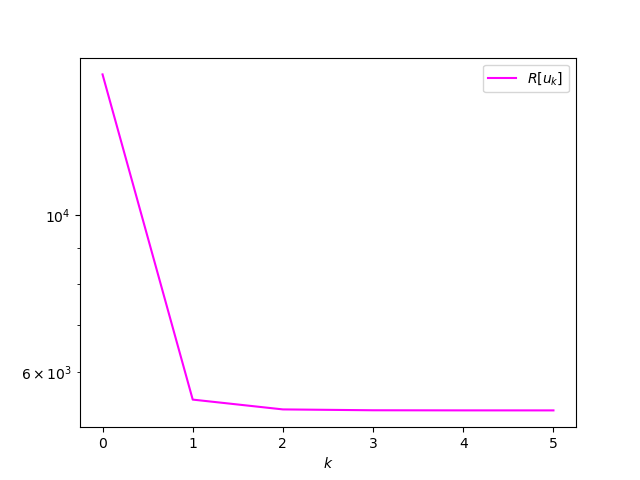}
\end{subfigure}%
\hspace*{\fill}
\begin{subfigure}{0.33\textwidth}
\includegraphics[width=\linewidth]{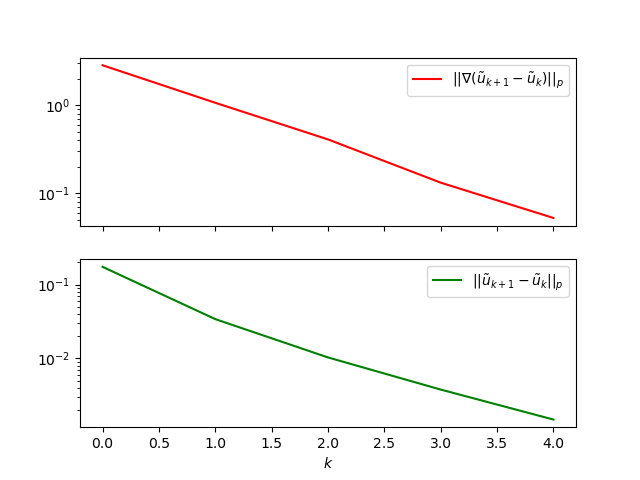}
\end{subfigure}%
\hspace*{\fill}
\begin{subfigure}{0.33\textwidth}
\includegraphics[trim=1.9cm 0.5cm 2.5cm 1.5cm, clip,width=\linewidth]{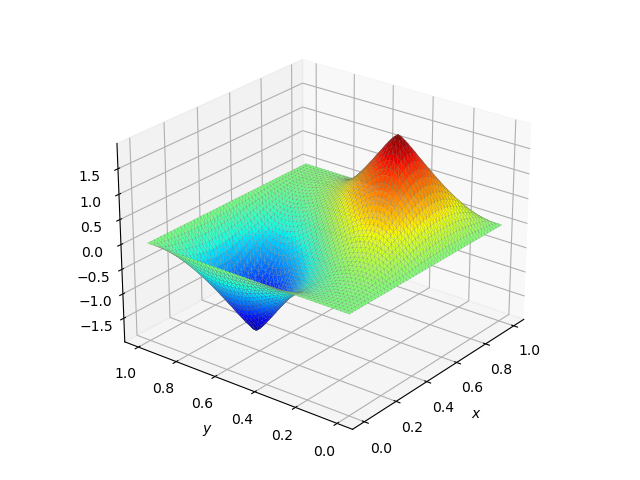}
\end{subfigure}%
\caption{Example~\ref{ex:square diag}: the detailed graphs for some particular values of $p$. Left column -- the Rayleigh quotients sequences $\ungulata{R[u_k]}_{k=0}^5$ (logarithmic $y$ scale); middle column -- the $\Wo$- and $L^p(\Omega)$-norms (red and green, respectively) of the difference between two consecutive approximations $u_k$ and $u_{k+1}$ (logarithmic $y$ scale); right column -- the graphs of the resulting approximations $u_5$. The rows correspond to (from the top to the bottom) $p = 1.6$, $p = 2$, $p = 3.5$, and $p = 5$.}
\label{fig:square diag detailed}
\end{figure}

\begin{figure}
\centering
\includegraphics[width=0.3\linewidth]{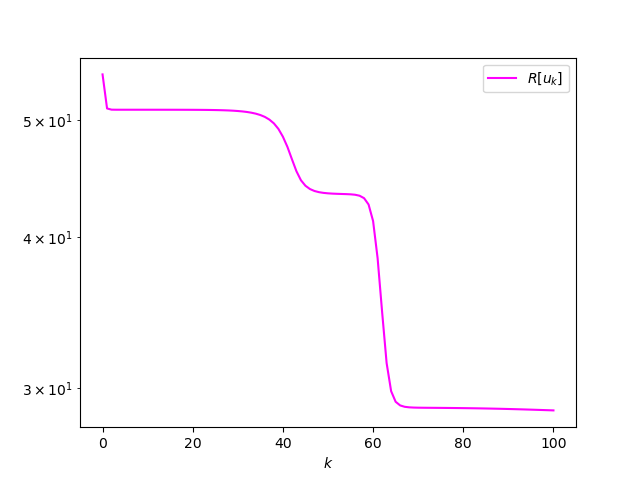}
\includegraphics[width=0.3\linewidth]{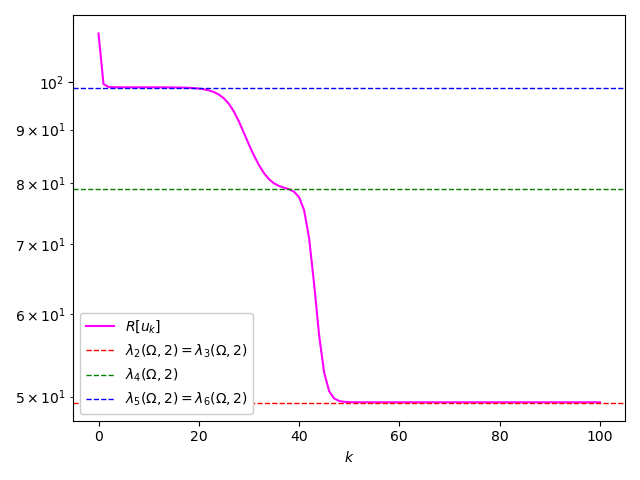}
\includegraphics[width=0.3\linewidth]{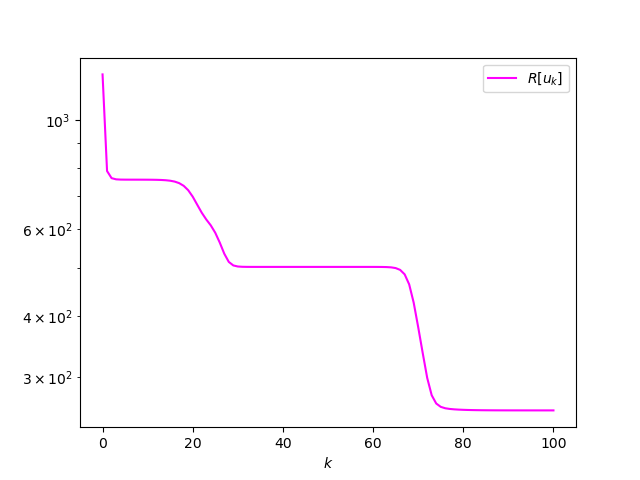}
\includegraphics[width=0.3\linewidth]{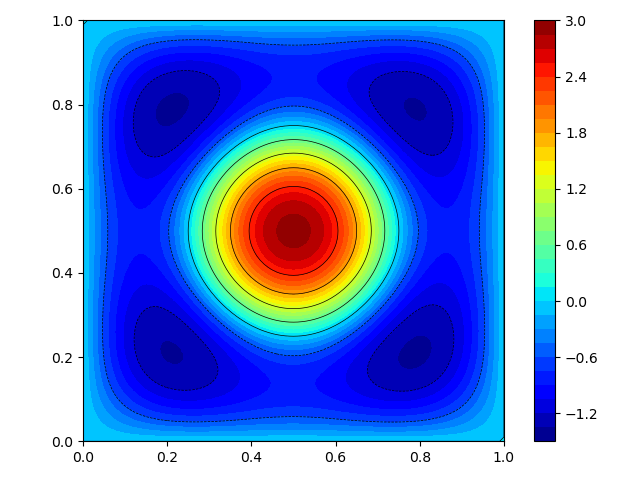}
\includegraphics[width=0.3\linewidth]{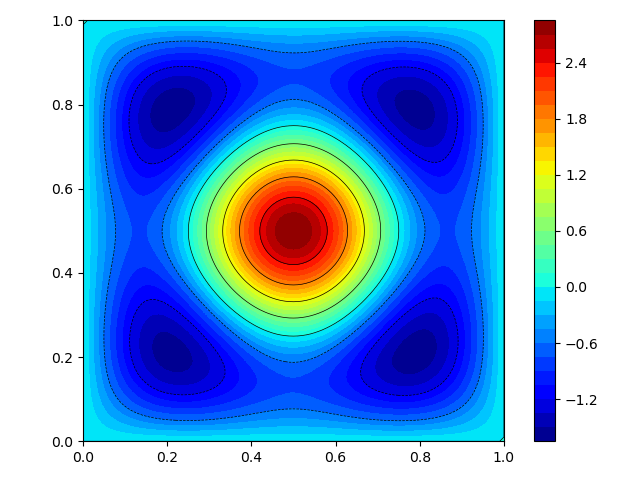}
\includegraphics[width=0.3\linewidth]{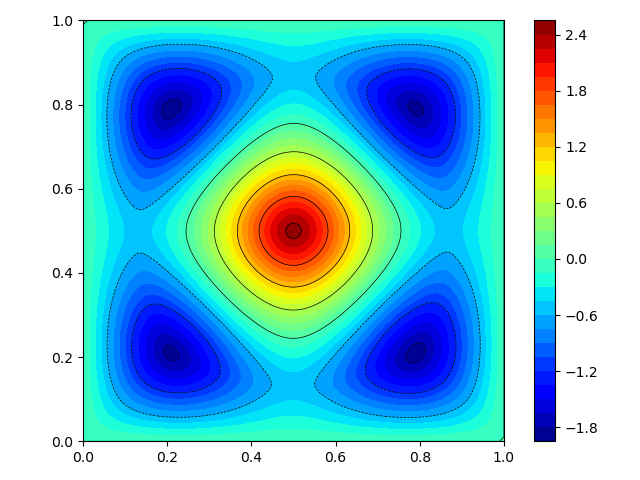}
\includegraphics[width=0.3\linewidth]{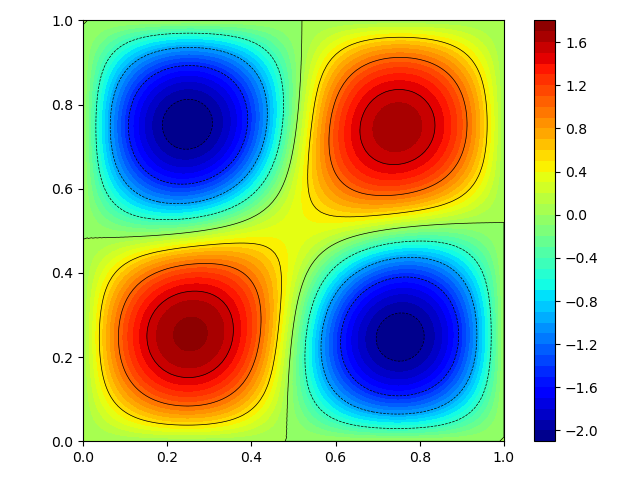}
\includegraphics[width=0.3\linewidth]{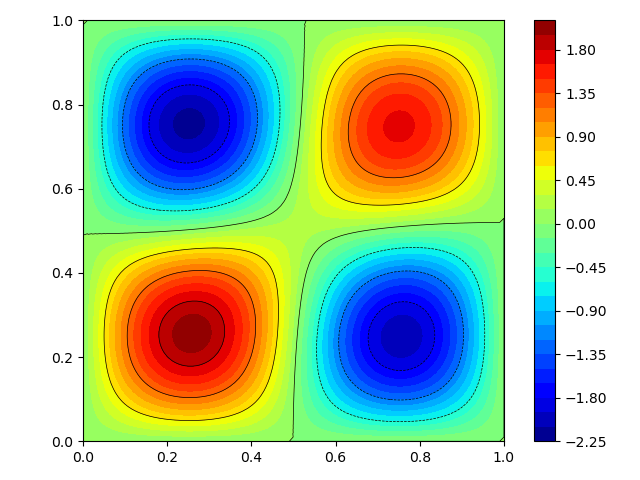}
\includegraphics[width=0.3\linewidth]{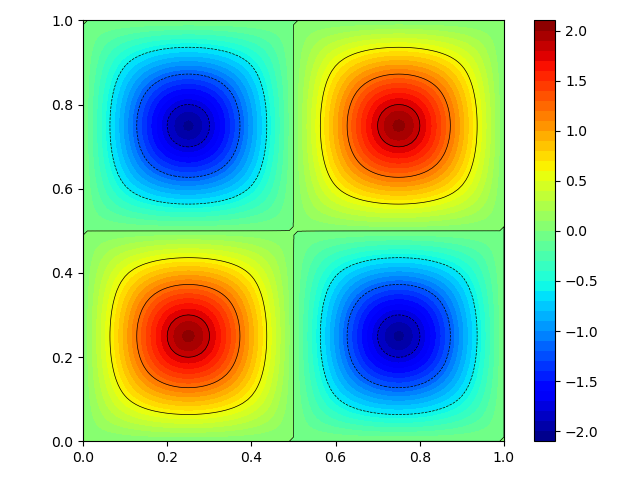}
\includegraphics[width=0.3\linewidth]{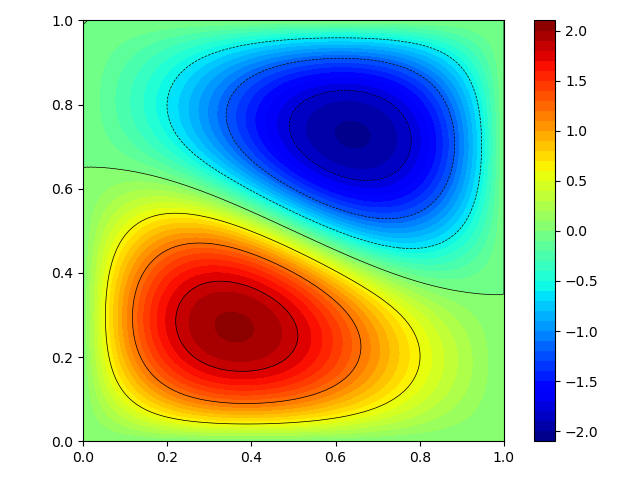}
\includegraphics[width=0.3\linewidth]{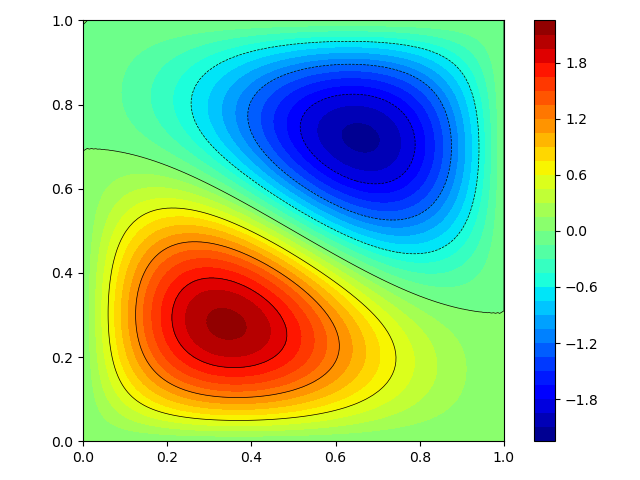}
\includegraphics[width=0.3\linewidth]{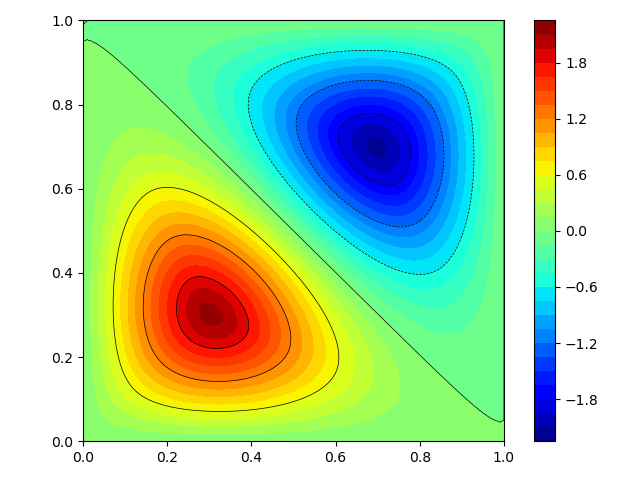}
\caption{Example~\ref{ex:square rad}: the results of the numerical experiments for $p = 1.7$ (left column), $p = 2$ (millde column), and $p = 3$ (right column). From the top to the bottom: the graphs of the Rayleigh quotients; the graphs of the functions $u_k$ corresponding to the first, the second, and the third stabilization stages of $R[u_k]$, respectively.}
\label{fig:square rad}
\end{figure}

\appendix
\section{Appendix. Basic properties of the $p$-Poisson problem}\label{sec:appendix}

The results of this section are well-known, and we provide details for the sake of clarity. 
At each step of \hyperref[alg]{Algorithm A}, the boundary value problem \eqref{phi} needs to be solved. This is a particular case of the more general $p$-Poisson problem
\begin{equation}\label{Dirichlet formal0} 
\begin{cases}
-\Delta_p v = f \text{ in } \Omega, \\
v\big|_{\partial \Omega} = 0,
\end{cases}
\end{equation}
where $v \in W_0^{1,p}(\Omega)$ and $f \in L^{p'}(\Omega)$. 
Recall from Section~\ref{sec:intro} that we always understand \eqref{Dirichlet formal0} in the weak sense, that is, $v$ must only satisfy
\begin{equation}\label{Dirichlet weak} 
\I{\Omega}{} \abs{\nabla v}^{p-2} \langle \nabla v, \nabla \testf \rangle \, dx = \I{\Omega}{} f \testf \,dx 
\quad \text{for any}~ \testf \in W_0^{1,p}(\Omega).
\end{equation}
It is known that the solution $v$ of \eqref{Dirichlet weak} exists, it is unique, and $v \neq 0$ if and only if $f \neq 0$.
The following simple result is essentially a weak maximum principle. 

\begin{lemma}\label{lem:sign-change lemma}
Let $f \in L^{p'}(\Omega)$ and let $v \in W_0^{1,p}(\Omega)$ be the solution of \eqref{Dirichlet weak}. 
If $\norm{v^+}_{p} > 0$, then $\norm{f^+}_{p'} > 0$, and if $\norm{v^-}_{p} > 0$, then $\norm{f^-}_{p'}>0$.
\end{lemma}
\begin{proof}
Suppose, without loss of generality, that $\norm{v^-}_{p} > 0$, but $\norm{f^-}_{p'}=0$, i.e., $f = f^+$. 
Taking $\testf = -v^-$ in \eqref{Dirichlet weak}, we get
$$ 
\norm{\nabla v^-}_{p}^p = -\I{\Omega}{} f  v^- \, dx\leqslant 0,
$$
which implies $\norm{v^-}_{p} = 0$, a contradiction. 
\end{proof}

Another useful property of the solution $v$ of \eqref{Dirichlet weak} is its continuous dependence in $\Wo$ on the right-hand side $f$.
\begin{lemma}\label{cont lemma}
Let $f_1, f_2 \in L^{p'}(\Omega)$ and let $v_1, v_2 \in W_0^{1,p}(\Omega)$ be the solutions of
\begin{equation}\label{cont}
\begin{cases}
-\Delta_p v_i = f_i \text{ in } \Omega, \\
v_i\big|_{\partial \Omega} = 0,
\end{cases}
\quad i = \overline{1;2}.
\end{equation} 
If $p \geqslant 2$, then 
\begin{equation}\label{cont dep est}
\norm{\nabla(v_1 - v_2)}_{p} 
\leqslant 
2^{\frac{p-2}{p-1}} \lambda_1^{\frac{-1}{p(p-1)}}(\Omega,p) \norm{f_1 - f_2}_{p'}^{\frac{1}{p-1}},
\end{equation}
and if $1 < p < 2$, then 
\begin{equation}\label{cont dep est p < 2}\norm{\nabla(v_1 - v_2)}_{p} 
\leqslant 
\frac{8{\lambda_1^{-1/p}(\Omega,p)}}{(p-1)} \prnth{\abs{\Omega}^{\frac{2-p}{2}} + 2\max_{i = \overline{1;2}}\norm{\nabla v_i}_p^{\frac{p(2-p)}{2}}}^{2/p} \, \norm{f_1 - f_2}_{p'}.
\end{equation}
\end{lemma}
\begin{proof} 
In the case $v_1 = v_2$, there is nothing to prove, so let $v_1 \ne v_2$. 
Taking $(v_1 - v_2)$ as test functions in the weak forms of \eqref{cont}, we get 
\begin{equation*}
\I{\Omega}{} \abs{\nabla v_i}^{p-2} \langle \nabla v_i, \nabla (v_1 - v_2) \rangle \, dx = \I{\Omega}{} f_i (v_1 - v_2) \,dx, \quad i = \overline{1;2},
\end{equation*}
and subtracting one equality from the other, we obtain
\begin{equation}\label{Otklonenije id}
\I{\Omega}{} \langle \abs{\nabla v_1}^{p-2} \nabla v_1 - \abs{\nabla v_2}^{p-2} \nabla v_2, \nabla (v_1 - v_2) \rangle \, dx = \I{\Omega}{}(f_1 - f_2) (v_1 - v_2) \, dx.
\end{equation} 
The integrand in the left-hand side of \eqref{Otklonenije id} can be estimated using the following well-known inequalities (see, e.g., \cite[pp.~97-100]{Lindqvist p-Laplace}): for any $a, b \in \mathbb{R}^D$,
\begin{align}
\label{Lindqvist >=2}\langle \abs{a}^{p-2} a - \abs{b}^{p-2}b, a - b \rangle \geqslant &\,2^{2-p} \abs{a - b}^p, &p \geqslant 2, \\
\label{Lindqvist <2}(1 + \abs{a}^2 + \abs{b}^2)^{\frac{2-p}{2}} \, \langle \abs{a}^{p-2} a - \abs{b}^{p-2}b, a - b \rangle \geqslant &\,(p-1) \abs{a - b}^2 , &1 < p \leqslant 2. 
\end{align}

Let $p \geqslant 2$. Estimating the left-hand side of \eqref{Otklonenije id} via \eqref{Lindqvist >=2} and applying the Hölder inequality in the right-hand side of \eqref{Otklonenije id}, we get
\begin{equation} 
2^{2 - p} \I{\Omega}{} \abs{\nabla (v_1 - v_2)}^p \,dx 
\leqslant \norm{v_1 - v_2}_{p} \, \norm{f_1 - f_2}_{p'}.
\end{equation}
From the variational characterization of $\lambda_1 (\Omega, p)$ it follows that
$$
2^{2-p} \norm{\nabla(v_1 - v_2)}_{p}^p \leqslant \lambda_1^{-1/p} (\Omega, p) \, \norm{\nabla (v_1 - v_2)}_{p} \, \norm{f_1 - f_2}_{p'}, 
$$
which finally leads to \eqref{cont dep est}.

Now consider $1 < p < 2$. In \eqref{Lindqvist <2}, we apply the estimate
$$ 
\prnth{1 + \abs{a}^2 + \abs{b}^2}^{\frac{2-p}{2}} \leqslant 2^{\frac{2-p}{2}} \prnth{1 + \prnth{\abs{a}^{2} + \abs{b}^2}^{\frac{2-p}{2}}} \leqslant 2^{\frac{2-p}{2}}\prnth{1 + 2^{\frac{2-p}{2}}\prnth{\abs{a}^{2-p} + \abs{b}^{2-p}}}.
$$
Since $(2-p)/2 < 1/2$, we get
\begin{equation}\label{eq:A81} 
\prnth{1 + \abs{a}^2 + \abs{b}^2}^{\frac{2-p}{2}} \leqslant 2 \prnth{1 +\abs{a}^{2-p}  + \abs{b}^{2-p}}.
\end{equation}
For $a = \nabla v_1(x)$ and $b = \nabla v_2(x)$, \eqref{Lindqvist <2} and \eqref{eq:A81}  give
\begin{align*} 
(p-1) \abs{\nabla (v_1 - v_2)}^2 \leqslant &\, 2\prnth{1 + \abs{\nabla v_1}^{2-p} + \abs{\nabla v_2}^{2-p}} \\
&\times\langle \abs{\nabla v_1}^{p-2} \nabla v_1 - \abs{\nabla v_2}^{p-2} v_2, \nabla (v_1 - v_2) \rangle.
\end{align*}
Raising the latter estimate to the power of $p/2$, applying the bound
$$
(1 + \abs{\nabla v_1}^{2-p} + \abs{\nabla v_2}^{2-p})^{p/2} \leqslant 2^{p}\prnth{1 + \abs{\nabla v_1}^{p(2-p)/2} + \abs{\nabla v_2}^{p(2-p)/2}}, 
$$
integrating the resulting inequality over $\Omega$, and applying the Hölder inequality with the exponents $2/p > 1$ and $(2/p)' = 2/(2-p)$, we deduce that
\begin{align} 
(p-1)^{p/2} \norm{\nabla(v_1 - v_2)}_{p}^p 
&\leqslant 
2^{{3p}/{2}} \quadr{\abs{\Omega}^{\frac{2-p}{2}}
+
\Sum{i=1}{2} \prnth{\I{\Omega}{} \abs{\nabla v_i}^{p} \,dx}^{\frac{2-p}{2}}} \\
&\times
\prnth{\I{\Omega}{} \langle \abs{\nabla v_1}^{p-2} \nabla v_1 - \abs{\nabla v_2}^{p-2} v_2, \nabla (v_1 - v_2) \rangle \,dx}^{p/2}
\end{align}
Employing \eqref{Otklonenije id} and the Hölder inequality, we conclude
$$ 
\norm{\nabla(v_1 - v_2)}_{p}^p \leqslant \frac{2^{3p/2}}{(p-1)^{p/2}}\prnth{\abs{\Omega}^{\frac{2-p}{2}} + \Sum{i=1}{2}\norm{\nabla v_i}_p^{\frac{p(2-p)}{2}}} 
\norm{v_1 - v_2}_{p}^{p/2} \norm{f_1 - f_2}_{p'}^{p/2}.
$$
Finally, estimating $\norm{v_1 - v_2}_{p}$ via \eqref{lambda_1 var}, dividing by $\norm{\nabla (v_1 - v_2)}_{p}^{p/2}$, and raising to the power of $2/p$, we derive \eqref{cont dep est p < 2}.
\end{proof}

\begin{corollary}\label{right sides convergence yields the solutions convergence}
Let $\ungulata{f_k}\subset L^{p'}(\Omega)$ converge strongly in $L^{p'}(\Omega)$ to some $f \in L^{p'}(\Omega)$, and for each $k \in \mathbb{N}$ let $v_k \in \Wo$ be the solution of 
\begin{equation}\label{eq:vkweak}
-\Delta_p v_k = f_k \text{ in } \Omega.
\end{equation}
Then $\ungulata{v_k}$ converges strongly in $W_0^{1,p}(\Omega)$ to the solution $v \in \Wo$ of 
$$ 
-\Delta_p v = f \text{ in } \Omega. 
$$
\end{corollary}
\begin{proof}
If $p \geqslant 2$, then the desired result follows directly from Lemma~\ref{cont lemma}. 
For $p \in (1;2)$, the same reasoning is valid provided $\sup \norm{\nabla v_k}_{p} = W < \infty$, since in this case by \eqref{cont dep est p < 2} we have
$$ 
\norm{\nabla (v - v_{k})}_{p} \leqslant C(\Omega,p) \cdot \prnth{\abs{\Omega}^{\frac{2-p}{2}} + 2\max\big\{\norm{\nabla v}_{p}, W\big\}^{\frac{p(2-p)}{2}}}^{2/p} \norm{f- f_{k}}_{p'} 
\quad \text{for any}~ k\in\mathbb{N}.
$$
Let us show the existence of $W$.
Using $v_k$ as a test function in \eqref{eq:vkweak}, applying the Hölder inequality, and employing \eqref{lambda_1 var}, we get
$$\norm{\nabla v_k}_{p}^p \leqslant \lambda_1^{-1/p}(\Omega,p) \, \norm{\nabla v_k}_{p} \norm{f_k}_{p'},$$ so that
$$ 
\norm{\nabla v_k}_{p} \leqslant \lambda_1^{\frac{-1}{p(p-1)}}(\Omega,p) \, \norm{f_k}_{p'}^{1/(p-1)}.
$$
Since $\ungulata{f_k}$ converges in $L^{p'}(\Omega)$, the existence of $W$ follows, completing the proof. 
\end{proof}

\begin{corollary}\label{strong partial limits of solutions}
Let $\ungulata{f_k} \subset W_0^{1,p}(\Omega)$ be bounded in $W_0^{1,p}(\Omega)$ and for each $k \in \mathbb{N}$ let $v_k \in \Wo$ be the solution of 
\begin{equation}\label{amano}
-\Delta_p v_k = |f_k|^{p-2}f_k \text{ in } \Omega.
\end{equation}
Choose a subsequence $\ungulata{f_{k_n}}$ converging to some $f \in W_0^{1,p}(\Omega)$ weakly in $W_0^{1,p}(\Omega)$.
Then $\ungulata{f_{k_n}}$ can be thinned out so that $\ungulata{v_{k_n}}$ converges strongly in $W_0^{1,p}(\Omega)$ to the solution of $-\Delta_p v = |f|^{p-2}f$ in $\Omega$.
\end{corollary}
\begin{proof} 
	Note that the existence of $f$ and $\ungulata{f_{k_n}}$ follows from the boundedness of $\{\norm{\nabla f_{k}}_{p}\}$. 
	By the Rellich-Kondrachov theorem, $\ungulata{f_{k_n}}$ can also be considered converging to $f$ strongly in $L^p(\Omega)$ and a.e.\ in $\Omega$.  
	Moreover, the Brezis-Lieb lemma allows to reduce $\ungulata{f_{k_n}}$ in such a way that $\ungulata{|f_{k_n}|^{p-2}f_{k_n}}$ converges to $|f|^{p-2}f$ strongly in $L^{p'}(\Omega)$. Then it follows from Corollary~\ref{right sides convergence yields the solutions convergence} that $\ungulata{v_{k_n}}$ converges strongly in $W_0^{1,p}(\Omega)$ to the function $v \in W_0^{1,p}(\Omega)$ solving $-\Delta_p v = |f|^{p-2}f$ in $\Omega$.
\end{proof}


\addcontentsline{toc}{section}{\refname}
\small
\begin{thebibliography}{99}

\bibitem{Anane}
Anane, J. A., \& Tsouli, N. (1996). On the second eigenvalue of the $p$-Laplacian. 
Nonlinear Partial Differential Equations (F\`es, 1994).  Pitman Research Notes Mathematics Series vol. 343, Longman, Harlow, 1996, pp.~1-9.

\bibitem{ACF}
Antonini, C. A., Ciraolo, G., \& Farina, A. (2023). Interior regularity results for inhomogeneous anisotropic quasilinear equations. Mathematische Annalen, 387(3), 1745-1776. \doi{10.1007/s00208-022-02500-x}

\bibitem{bessa}
Bessa, G. P., Gimeno, V., \& Jorge, L. (2019). Green functions and the Dirichlet spectrum. Revista Matem\'atica Iberoamericana, 36(1), 1-36. \doi{10.4171/RMI/1119}


\bibitem{Biezuner}
Biezuner, R. J., Ercole, G., \& Martins, E. M. (2009). Computing the first eigenvalue of the $p$-Laplacian via the inverse power method. Journal of Functional Analysis, 257(1), 243-270. \doi{10.1016/j.jfa.2009.01.023}

\bibitem{Bobkov}
Bobkov, V. (2014). Least energy nodal solutions for elliptic equations with indefinite nonlinearity. Electronic Journal of Qualitative Theory of Differential Equations, 2014(56), 1-15.
\doi{10.14232/ejqtde.2014.1.56}

\bibitem{Bozorgnia1}
Bozorgnia, F. (2016). Convergence of inverse power method for first eigenvalue of $p$-Laplace operator. Numerical Functional Analysis and Optimization, 37(11), 1378-1384. \doi{10.1080/01630563.2016.1211682}

\bibitem{Bozorgnia2}
Bozorgnia, F., \& Arakelyan, A. (2023). Approximation of the second eigenvalue of the $p$-Laplace operator in symmetric domains. Journal of Computational and Applied Mathematics, 434, 115349. \doi{10.1016/j.cam.2023.115349}

\bibitem{BrFr1}
Brasco, L., \& Franzina, G. (2013). On the Hong–Krahn–Szego inequality for the $p$-Laplace operator. Manuscripta Mathematica, 141(3), 537-557. \doi{10.1007/s00229-012-0582-x}

\bibitem{brounreichel}
Brown, B. M., \& Reichel, W. (2002). Computing eigenvalues and Fu\v{c}ik-spectrum of the radially symmetric $p$-Laplacian. Journal of Computational and Applied Mathematics, 148(1), 183-211.
\doi{10.1016/S0377-0427(02)00581-2}

\bibitem{Cuesta 1}
Cuesta, M., De Figueiredo, D., \& Gossez, J. P. (1999). The beginning of the Fu\v{c}ik spectrum for the $p$-Laplacian. Journal of Differential Equations, 159(1), 212-238. \doi{10.1006/jdeq.1999.3645}

\bibitem{DiBenedetto}
DiBenedetto, E. (1983). $C^{1+\alpha}$ local regularity of weak solutions of degenerate elliptic equations. Nonlinear Analysis: Theory, Methods \& Applications, 7(8), 827-850. \doi{10.1016/0362-546X(83)90061-5}

\bibitem{diening}
Diening, L., Lindqvist, P., \& Kawohl, B. (2013). Mini-Workshop: the $p$-Laplacian operator and applications. Oberwolfach Reports, 10(1), 433-482.
\doi{10.4171/OWR/2013/08}

\bibitem{DrRob}
Dr\'abek, P., \& Robinson, S. B. (1999). Resonance problems for the $p$-Laplacian. Journal of Functional Analysis, 169(1), 189-200. \doi{10.1006/jfan.1999.3501}

\bibitem{Drabek on Courant}
Dr\'abek, P., \& Robinson, S. B. (2002). On the generalization of the Courant nodal domain theorem. Journal of Differential Equations, 181(1), 58-71. \doi{10.1006/jdeq.2001.4070}

\bibitem{ercole}
Ercole, G. (2018). Solving an abstract nonlinear eigenvalue problem by the inverse iteration method. Bulletin of the Brazilian Mathematical Society, New Series, 49, 577-591. \doi{10.1007/s00574-018-0070-3}

\bibitem{geodesic}
Gimeno, V., \& Sarrion-Pedralva, E. (2022). First eigenvalue of the Laplacian of a geodesic ball and area-based symmetrization of its metric tensor. Journal of Mathematical Inequalities, 16(1). \doi{10.7153/jmi-2022-16-28}

\bibitem{Suomalaiset}
Heinonen, J., Kipelainen, T., \& Martio, O. (2006). Nonlinear potential theory of degenerate elliptic equations. Dover Publications.

\bibitem{Horak}
Hor\'ak J. (2011). Numerical investigation of the smallest eigenvalues of the $p$-Laplace operator on planar domains. Electronic Journal of Differential Equations, 2011(132), 1-30. 
\url{https://ejde.math.txstate.edu/Volumes/2011/132/horak.pdf}

\bibitem{HMP}
Hurtado, A., Markvorsen, S., \& Palmer, V. (2016). Estimates of the first Dirichlet eigenvalue from exit time moment spectra. Mathematische Annalen, 365(3), 1603-1632. \doi{10.1007/s00208-015-1316-7}

\bibitem{doubly nonlinear}
Hynd, R., \& Lindgren, E. (2016). A doubly nonlinear evolution for the optimal Poincaré inequality. Calculus of Variations and Partial Differential Equations, 55 (100), 1-22. \doi{10.1007/s00526-016-1026-3}

\bibitem{Hynd}
Hynd, R., \& Lindgren, E. (2016). Inverse iteration for $p$-ground states. Proceedings of the American Mathematical Society, 144(5), 2121-2131. \doi{10.1090/proc/12860}

\bibitem{Hynd FA}
Hynd, R., \& Lindgren, E. (2017). Approximation of the least Rayleigh quotient for degree $p$ homogeneous functionals. Journal of Functional Analysis, 272(12), 4873-4918. \doi{10.1016/j.jfa.2017.02.024}

\bibitem{LeftonWei}
Lefton, L., \& Wei, D. (1997). Numerical approximation of the first eigenpair of the $p$-Laplacian using finite elements and the penalty method. Numerical Functional Analysis and Optimization, 18(3-4), 389-399.
\doi{10.1080/01630569708816767}

\bibitem{Lindquist}
Lindqvist, P. (2008). A nonlinear eigenvalue problem. In Topics In Mathematical Analysis (pp. 175-203). \doi{10.1142/9789812811066\_0005}

\bibitem{Lindqvist p-Laplace}
Lindqvist, P. (2019). Notes on the stationary $p$-Laplace equation. Berlin: Springer International Publishing. \doi{10.1007/978-3-030-14501-9}

\bibitem{Lou}
Lou, H. (2008). On singular sets of local solutions to $p$-Laplace equations. Chinese Annals of Mathematics, Series B, 29(5), 521-530. \doi{10.1007/s11401-007-0312-y}

\bibitem{brownian}
McDonald, P. (2002). Isoperimetric conditions, Poisson problems, and diffusions in Riemannian manifolds. Potential Analysis, 16(2), 115-138. \doi{10.1023/A:1012638112132}

\bibitem{MM1}
McDonald, P., \& Meyers, R. (2003). Dirichlet spectrum and heat content. Journal of Functional Analysis, 200(1), 150-159. \doi{10.1016/S0022-1236(02)00076-9}

\bibitem{PS}
Pucci, P., \& Serrin, J. B. (2007). The maximum principle. Springer.
\doi{10.1007/978-3-7643-8145-5}

\bibitem{sat}
Sattinger, D. H. (1972). Monotone methods in nonlinear elliptic and parabolic boundary value problems. Indiana University Mathematics Journal, 21(11), 979-1000.
\doi{10.1512/iumj.1972.21.21079}

\bibitem{Sun}
Sun, J., \& Zhou, A. (2016). Finite element methods for eigenvalue problems. Chapman and Hall/CRC. \doi{10.1201/9781315372419}

\bibitem{tolksdorf}
Tolksdorf, P. (1984). Regularity for a more general class of quasilinear elliptic equations. Journal of Differential equations, 51(1), 126-150. \doi{10.1016/0022-0396(84)90105-0}

\bibitem{bergbucur}
van den Berg, M., \& Bucur, D. (2020). Sign changing solutions of Poisson's equation. Proceedings of the London Mathematical Society, 121(3), 513-536. \doi{10.1112/plms.12334}

\bibitem{git}
\url{https://github.com/gtmrxn/InvIt.git}

\end{thebibliography}
\end{document}